\documentclass[preprint]{imsart}
\RequirePackage[OT1]{fontenc}
\RequirePackage{amsthm,amsmath}
\RequirePackage[numbers]{natbib}
\RequirePackage[colorlinks,citecolor=blue,urlcolor=blue]{hyperref}

\pubyear{\today}

\startlocaldefs
\usepackage[tikz]{bclogo}
\usepackage{algorithm2e}

\usepackage{amssymb, amsmath, natbib, amsthm, graphicx, euscript, mathrsfs, enumerate, empheq}

\newcommand{\eqd}{\stackrel{Law}{=}}

\newcommand{\T}{\mathcal{T}}
\newcommand{\E}{{\mathbb E}}

\renewcommand{\i}{\mathrm{i}}
\newcommand{\hs}{\hat{s}}
\newcommand{\too}{\rightsquigarrow}
\newcommand{\C}{\mathcal{C}}

\def\R{I\!\!R}
\def\N{I\!\!N}
\newcommand{\eps}{\varepsilon}

\newtheorem{thm}{Theorem}[section]
\newtheorem{lem}[thm]{Lemma}
\newtheorem{cor}[thm]{Corollary}
\newtheorem{prop}[thm]{Proposition}
\newtheorem{rem}[thm]{Remark}

\renewcommand{\P}{{\mathbb P}}
\renewcommand{\L}{\mathscr{L}}

\newcommand{\B}{\mathcal{B}}
\newcommand{\W}{\mathcal{W}}

\newcommand{\bd}{\bar{\delta}}

\renewcommand{\kappa}{\varkappa}
\def\sec#1{\underline{\textbf{#1}}}
\newcommand{\p}{\Upsilon}
\renewcommand{\k}{k}
\newcommand{\Z}{\mathscr{Z}}
\renewcommand{\wp}{\widetilde{\p}}

\newcommand{\bz}{\breve{\zeta}}
\newcommand{\bp}{\breve{\p}}
\newcommand{\btau}{\breve{\tau}}
\newcommand{\cc}{\breve{c}}
\newcommand{\DD}{\mathcal{D}}
\newcommand{\DDD}{\mathscr{D}}
\newcommand{\I}{{\mathbb I}}
\newcommand{\tzeta}{\tilde{\zeta}}
\newcommand{\F}{\mathcal{F}}

\newlength\fboxseph
\newlength\fboxsepv

\makeatletter

\def\longboxed#1{\leavevmode\setbox\@tempboxa\hbox{\color@begingroup%
\kern\fboxseph{\m@th$\displaystyle #1 $}\kern\fboxseph%
\color@endgroup }\my@frameb@x\relax}

\def\my@frameb@x#1{%
\@tempdima\fboxrule \advance\@tempdima \fboxsepv \advance\@tempdima \dp\@tempboxa\hbox {%
\lower \@tempdima \hbox {%
\vbox {\hrule\@height\fboxrule \hbox{\vrule\@width\fboxrule #1 \vbox{%
\vskip\fboxsepv \box\@tempboxa \vskip\fboxsepv}#1 \vrule\@width\fboxrule }%
\hrule \@height \fboxrule }}}}

\makeatother

\begin{document}
\begin{frontmatter}
\title{Convergence rates of maximal deviation distribution for projection estimates of L\'evy densities}
\runtitle{Convergence rates for L{\'e}vy densities}

\begin{aug}
\author{\fnms{Valentin} \snm{Konakov}\ead[label=e1]{VKonakov@hse.ru}}
\and
\author{\fnms{Vladimir} \snm{Panov}\ead[label=e2]{VPanov@hse.ru}}

\address{Laboratory of Stochastic Analysis and its Applications \\National Research University Higher School of Economics\\ 
Shabolovka  31, building G,  115162 Moscow, Russia\\
\printead{e1,e2}}

\runauthor{V.Konakov and V.Panov}
\affiliation{Higher School of Economics}
\end{aug}

\begin{abstract}
In this paper, we consider projection estimates for L{\'e}vy densities in high-frequency setup.  We give a unified treatment for different sets of basis functions and focus on the asymptotic properties of the maximal deviation distribution for these estimates. Our results are  based on the  idea to reformulate the problems in terms of Gaussian processes of some special type and to further analyze these Gaussian processes. 
In particular, we construct a sequence of excursion sets, which guarantees the convergence of the deviation distribution to the Gumbel distribution.   We show that the rates of convergence presented in previous articles on this topic are logarithmic and construct the sequences of accompanying laws, which approximate the deviation distribution with polynomial rate.
\end{abstract}



\tableofcontents
\end{frontmatter}

\section{Introduction}
\label{intro}
Consider a one-dimensional L{\'e}vy process \(X_{t}\) with L{\'e}vy triplet \(\left( \mu, \sigma, \nu \right)\).  Assume that measure \(\nu\) has density \(s(\cdot)\), known as L{\'e}vy density,  that is,
\[
	\nu(B) = \int_{B} s(u) du, \quad \forall B \in \B(\R),
\] 
Assuming that some discrete equidistant observations \(X_{0}, X_{\Delta}, ..., X_{n \Delta}\) of the process \(X_{t}\) are available, it is natural to ask how  one can  statistically infer on the L{\'e}vy density \(s(\cdot)\), or more generally speaking, on the L{\'e}vy measure \(\nu\).  The answer to this question highly depends on the type of  the available data.  The first situation, known as \textit{high-frequency setup}, is based on the assumption that the time distance between the observations \(\Delta=\Delta_{n}\)  depends on \(n\) and tends to \(0\) as \(n \to \infty\). Moreover, very often (and in this paper) it is also assumed that the time horizon \(T=n \Delta_{n}\to \infty\) as \(n \to \infty\). Non-parametric inference for this case has been  considered by Comte and Genon-Catalot (\citeyear{CGC}),  Figueroa-L{\'o}pez (\citeyear{Fig3}), Figueroa-L{\'o}pez and Houdr{\'e} (\citeyear{FLH}). The second situation, the so-called \textit{low-frequency setup}, in which \(\Delta\) is fixed, has been extensively studied by 
Nickl and Reiss (\citeyear{NicklReiss}), Gugushvili (\citeyear{Gugu}), Belomestny (\citeyear{Belomest2011}), Comte and Genon-Catalot (\citeyear{CGC2}), Chen, Dalaigle and Hall (\citeyear{CDH}), Neumann and Reiss (\citeyear{NeuReiss}), van Es, Gugushvili and Spreij (\citeyear{EGS}). The essential idea in almost all papers mentioned above is to express the L{\'e}vy measure in terms of the characteristic function of \(X_t\) and then replace this characteristic  function by its natural nonparametric estimator. For instance, applying the L\'evy-Khintchine formula for the characteristic function 
\( \phi_{X_{t}}(u) = \E \left[\exp\left\{ 
		\i u X_t
\right\}\right], \)
\[
\phi_{X_{t}}(u) 
=\exp \left\{ 
	t \left( 
  		  \i \mu u-\frac{1}{2}\sigma^2 u^2
  		  +
   		 \int_{\mathbb{R}\setminus \{0\}}\left( e^{\i u x}-1 - \i u x 		\cdot\mathbf{1}_{\{ |x|\leq 1 \}} \right)\nu(dx)
 \right)
    \right\}\]
 to the compound Poisson process without a drift (that is,  \(\mu=\sigma=0,\) and  L{\'e}vy messure is finite, \(\int_{\R}\nu(du) <\infty\)), we get that the Fourier transform of the function \(x \cdot s(x)\), 
is equal to 
\[
	\F_{x \cdot s(x)}(u) = \int_{\R }e^{\i u x} x s(x) dx = -i  \frac{\phi'_{X_{\Delta}} (u)}
	{
		\Delta \phi_{X_{\Delta}}(u)
	}.
\]
Substituting into the last formula natural estimator for the characteristic function 
\(\hat{\phi}_{X_{\Delta}}(u) := n^{-1} \sum_{k=1}^{n} \exp\left\{ 
	\i u \left(
		X_{k \Delta} - X_{(k-1) \Delta}
	\right)
\right\}
\)
instead of \(\phi_{X_{\Delta}}(u)\), and avoiding the situation when \(\hat{\phi}_{X_{\Delta}}(u)=0\) by multiplying the fraction by the indicator function \(\I\{|\hat{\phi}_{X_{\Delta}}(u)|>A\}\) with some \(A\), we finally get that the function
\[
	\widehat{\F}_{x \cdot s(x)}(u)  =  \frac{\sum_{k=1}^{n} 
	\left(   
		X_{k \Delta} - X_{(k-1) \Delta}
	\right) e^{
	\i u \left(
		X_{k \Delta} - X_{(k-1) \Delta}
	\right)
}
	}
	{
		\Delta \cdot \sum_{k=1}^{n} e^{
				\i u \left(
					X_{k \Delta} - X_{(k-1) \Delta}
				\right)
		}
	} \I\{|\hat{\phi}_{X_{\Delta}}(u)|>A\}
\]
is a reasonable estimator for \(\F_{x \cdot s(x)}(u) \),   and can be further used for statistical inference on \(s(u)\). Similar ideas are widely used in papers on non-parametric inference, but we focus on some other approaches in the current research.

It is a worth mentioning that in most papers on this topic, the quality of proposed estimator for \(s(\cdot)\) is measured in terms of quadratic risk. More precisely, for a fixed estimate \(\hat{s}_{n}^{\circ} (x)\), a collection of L{\'e}vy processes \(\T\) and    a window \(D = [a,b] \subset \R / \{0\}\), it is common to prove two statements, which present  upper and lower bounds for the difference between \(\hat{s}_{n}^{\circ} (x)\) and the true density function \(s(x)\). These two statements are usually formulated as follows:
\begin{eqnarray*}
 	\sup_{\T}
 	\E \left(
 	\hat{s}_{n}^{\circ} (x) - s(x)
	\right)^{2}
	&\lesssim& f(n), \qquad  \forall x \in D,\\
	\inf_{\left\{\hat{s}_{n} (x) \right\}} 	\sup_{\T}
 	\E \left(
 	\hat{s}_{n}(x) - s(x)
	\right)^{2}
	&\gtrsim& g(n), \qquad \forall x \in D,
\end{eqnarray*}
where by \(\left\{ \hat{s}_{n}(x) \right\}\) we denote the set of all estimates of the L{\'e}vy density \(s(x)\), and \(f(n), g(n)\) are two functions  tending to 0 as \(n \to \infty\). If \(f(n) \approx g(n)\),  it is usually claimed that the estimate  \(\hat{s}_{n}^{\circ} (x)\) posseses the optimality property.

In this research, we focus on another aim -   we analyze the maximal deviation distribution of  the projection estimator \(\hs_{n}(x)\) defined below by \eqref{hsdef}. 
More precisely,  we are interested in the asymptotic properties of the distribution function of 
\begin{eqnarray}
\label{DDD}
	\DDD_{n}:=		\sup_{x \in D} \left(	
			\frac{
				\left|
					\hs_{n} (x)  - s (x)  
				\right|
			}
			{	
				\sqrt{s(x)}
			}
		\right).
\end{eqnarray}
 To the best of our knowledge, the unique research  in this direction is provided by  Figueroa-L{\'o}pez (\citeyear{Fig3}), who considered the maximal deviation distribution  for projection estimates to the space spanned by Legendre polynomials of orders 0 and 1. We emphasize the main differences between our paper and the paper by  Figueroa-L{\'o}pez (\citeyear{Fig3}) later  in Section~\ref{disc}. For the moment, let us only mention that our setup covers more general classes of estimates - in particular, we provide the proof for Legendre polynomials of any order, as well as for trigonometric basis and wavelets. 

One of the main sources of our inspiration is the paper by Konakov and Piterbarg (\citeyear{KP}), where the asymptotics of the maximal deviation distribution is proven for the kernel estimates of regression functions. Konakov and Piterbarg (\citeyear{KP}) showed that the convergence to asymptotic distribution given in Bickel and Rosenblatt (\citeyear{Bickel}) is very slow (of logarithmic order) and this rate cannot be improved. Moreover, in that paper, it is obtained a sequence on accompanying laws with power rate of convergence.  Nevertheless, the regression problem completely differs from the estimation of L{\'e}vy density, and therefore we are not able to apply even the techniques from that research. 

The contribution of this paper is twofold. First, we derive the asymptotic behaviour of the maximal deviation distribution for a broad class of projection estimates of the L{\'e}vy density. This result can be further applied for  constructing confidence intervals and statistical tests.  Second,  we show that the rates of convergence given in Figueroa-L{\'o}pez (\citeyear{Fig3}) are of logarithmic order and provide the sequences of accompanying laws with power rates. 

The paper is organized as follows. In the next section, we explain our setup and our assumptions on the set of basis functions. Section~\ref{mainres} contains a collection of our results. Later on, in Sections~\ref{statsec} - \ref{wav}  we prove and discuss these results separately for different choices of basis functions - trigonometric functions,  Legendre polynomials and wavelets. Some further discussions of our contribution to this topic can be found in Section~\ref{disc}. Additional proofs are given in the Appendix.

\section{Set-up}
\label{basis}
\sec{Collections of basis functions.} In this paper, we follow the set-up from \cite{Fig3}, and study the estimation of the L{\'e}vy density \(s(x)\) over a window   \(D\), based on discrete observations of the process on an interval \([0, T]\).  We consider  a family of finite linear combinations of functions from orthonormal collection  \(\left\{ \varphi_{r}(x): D \to \R, \; r=1..d\right\}\):
\begin{eqnarray}
\label{sieves}
\L= 
\Biggl\{
	\sum_{r=1}^{d} \beta_{r} \varphi_{r}(x), \qquad \vec{\beta}= \left( \beta_{1}, ... , \beta_{d} \right) \in \R^{d}
\Biggr\},
\end{eqnarray}
and later project the L{\'e}vy density to the space \(\L\) in  \(L^{2}\) - metric. 
%We will follow the original paper by Figueroa-L{\'o}pez (\citeyear{Fig3}) and call such models sieves, but keep in mind that this is actually a partial case of the more general notion, which can be found for instance in Beder (\citeyear{Beder}). %Common examples of  models in the form \eqref{sieves} are  splines, wavelets and trigonometric polynomials.  
In this article, we do not restrict our attention to a particular class of collections \(\{\varphi_{r}(x)\}\), but assume that for any \(m \in \N\) there exists a set of normalized bounded functions \(\left\{ \psi_{j}^m: D \to \R \right\}_{j=0}^{J}\)  supported on \([a, a+\delta)\), where \(\delta = (b-a)/m\), such that 
\begin{eqnarray*}
	\label{phir}
	\Biggl\{ \varphi_{r}(x) , \; r=1..d \Biggr\} &=&
	\Biggl\{  \psi_{j}^m  \left(  x -  \delta (p-1)  \right) I\left\{ x \in I_{p} \right\}, 
  j=0..J,\; p=1..m \Biggr\}, \\
  && 
	\nonumber \hspace{1.5cm}
	\qquad \mbox{where} \qquad I_{p}:= [ a+\delta(p-1), 
	a+\delta p ).
\end{eqnarray*}
 In what follows, it is important how the basis functions \(\psi_{j}^{m}(x)\) depend on \(m\), or in other words, how these functions depend on \(\delta=(b-a)/m\).  Below we give an intuition about the dependence. Note that in most examples,  basis on \([a, a+\delta)\) is constructed from a basis \(\{\widetilde{\psi}_j (x)\}_{j=0..J}\) on some ``standard'' interval  \([\tilde{a}, \tilde{b}]\) by changing the variables: 
\begin{eqnarray}
\label{psij}
\psi_j^m (x) = \sqrt{\frac{\tilde{b}-\tilde{a}}{\delta}} \cdot
\widetilde{\psi}_j \left(
	\frac{
		(\tilde{b} - \tilde{a}) (x-a)
	}{
		\delta 
	}+  \tilde{a}
\right),
\end{eqnarray}
and therefore \(\psi_j^m (x) = O(\sqrt{m})\) as \(m \to \infty\). Some typical examples are listed below.
\begin{enumerate}[(i)]
	\item Trigonometric basis on \([a,a+\delta)\)
	\begin{eqnarray*}
	\Bigl\{ 
		\psi_{j}^{m}(x), \; j=0..J
	\Bigr\} 
	&=&
	\Bigl\{  
	\chi_{0}(x) = \frac{1}{\sqrt{\delta}}, \; \chi_{j}(x) = \sqrt{\frac{2}{\delta}} \cos\left( 
		2 j \pi (x-a) / \delta
	\right),\\
	&& 
	\tilde\chi_{j}(x) = \sqrt{\frac{2}{\delta}}\sin\left( 
		2 j \pi (x-a) / \delta
	\right),
	\; j=1 .. (J/2)
	\Bigr\}.
	\end{eqnarray*}
with even \(J\).	In this case, it is natural to define  the ``standard'' interval as \([\tilde{a}, \tilde{b}] = [0, 2\pi]\), and basis on this interval as 
\[ \Bigl\{\widetilde{\psi}_j (x)\Bigr\} =
	\Bigl\{\
		\frac{1}{\sqrt{2 \pi}} , \quad \sqrt{2} \cos(jx), \quad \sqrt{2} \sin(jx)
	\Bigr\}.
\]
	\item Legendre polynomials, that is 
	\[
		\psi_{j}^{m} (x) = \sqrt{\frac{2 j +1 } {\delta}} P_{j}\Bigl( 
		\left(
			(x - a - \delta) + (x - a  ) 
		\right)/ \delta
		\Bigr),
	\]
	where 
	\[
		P_{j}(x)= \frac{1}{j! 2^{j}}\left[ \left( x^{2}-1\right) ^{j}\right] ^{(j)},
		\qquad j=0..J	
	\] are  Legendre polynomials on \([-1,1]\).  The set of orthonormal polynomials on \([\tilde{a}, \tilde{b}]=[-1,1]\)
\[ \Bigl\{\widetilde{\psi}_j (x)\Bigr\}
=
	\Bigl\{ \sqrt{(2 j +1) / 2 } \cdot P_j (x), \quad j=0..J  \Bigr\},
\] 
plays the role of standard basis.
	\item Wavelets, for instance Haar wavelets  
	\begin{eqnarray*}
		\psi_{j}^{m} (x)  =\Biggl\{  \frac{1}{\sqrt{\delta}}, \qquad \frac{1}{\sqrt{\delta}}\Biggl(I\{x \in [a+\delta/2, a+\delta ]\} - I\{x \in [a,a+\delta/2]\} \Biggr)
		\Biggr\},
	\end{eqnarray*}
where \(\delta\) is usually taken as \(2^{-l}\) for some \(l \in \N\). This role of standard interval is  usually given to the interval \([\tilde{a}, \tilde{b}]=[0,1]\) supplied with two functons
\[\Bigl\{\widetilde{\psi}_j (x)\Bigr\} 
=
	\Bigl\{ 
		 1, \quad  I\{x \in [1/2, 1]\} - I\{x \in [0,1/2]
	\Bigr\} .
\]
\end{enumerate}
To sum up, basis functions typically depend on \(\delta\) as it is given by \eqref{psij}, where the function \(\tilde{\psi}_{j}\) are bounded and supported on some compact \( [\tilde{a}, \tilde{b}]\).
For theoretical studies, we assume that the function \(\sqrt{\delta}\psi_{j}^m (x)\) and its total variation are bounded by some absolute constants \(\C_{1}\) and \(\C_{2}\), that is, for all \(j,m\),
\begin{eqnarray}\label{cond1}
	\sqrt{\delta} \cdot \sup_{x \in I_{1}}|\psi_{j}^m (x) | \leq \C_{1}, \qquad	\sqrt{\delta}\cdot V_{a}^{a+\delta}
	(
		\psi_{j}^m
	)
	\leq  \C_{2},
\end{eqnarray}
where by \(V_{a}^{a+\delta}
	(
		\psi_{j}^m
	)
	\) we denote the total variation of the function \(\psi_{j}^m\), 
	\[
	V_{a}^{a+\delta}
	(
		\psi_{j}^m
	)
	:= \sup_{
			\|P\| \to 0
		} 
		\sum_{i=1}^{n} 
			\left| 
				\psi_{j}^m (x_{i}) - \psi_{j}^m (x_{i-1})
			\right|,\]
\(P\) ranges over the partitions \(a=x_{0}<x_{1}<...<x_{n}=a+\delta\) equipped with the norm \(\|P\|=\max_{i}|x_{i}-x_{i-1}|\).

\sec{Projection estimates.} Consider the \(L^{2}\)-scalar product and \(L^{2}\)-norm in the space of functions \(\left\{ g: D \to \R\right\}\), and introduce the estimator \(\tilde{s}(x)\) as  the orthogonal projection of the funciton \(s(x)\) on \(\L\) with respect to this norm: 
\begin{eqnarray}
\label{tildes}
	\tilde{s}(x) := \sum_{r=1}^{d} \beta_{r} \varphi_{r}(x),
\end{eqnarray} where 
\[
	\beta_{r} =  \beta(\varphi_{r}) = \int_{D} \varphi_{r}(x) s(x) dx = \int_{D} \varphi_{r}(x) \tilde{s}(x) dx.
\]
Returning to the statistical problem, that is, to the problem of statistical estimation  of \(\tilde{s}(x)\) by the  equidistant observations \(X_{0}, X_{\Delta}, ..., X_{n \Delta}\), we realize that the main difficulty consists in estimation of  \(\beta(\varphi)\) for different basis functions \(\varphi\). As it was explained earlier, there exists a crucial difference  in the assumptions on the design. It turns out, that in case of the low-frequency setup, this question is not well-understood in the literature.  As for the high-frequency setup, estimation of \(\beta(\varphi)\)  has been extensively studied  in \cite{FL4} and  \cite{Woerner}, where it is shown that the coefficients of \(\beta_{r}\) can be estimated by 
\begin{eqnarray}
\label{hb}
	\hat\beta (\varphi_{r}) = \frac{1}{n \Delta} \sum_{k=1}^{n} 
	\varphi_{r} \left( X^{(k)}_{\Delta} \right), \qquad \mbox{where} \quad
	X^{(k)}_{\Delta} = X_{k\Delta} - X_{(k-1)\Delta}.
\end{eqnarray}
Next, we can plug in the estimator \(\hat\beta (\varphi_{r})\) in \eqref{tildes}, and get that 
\begin{eqnarray}
\label{hsdef}
	\hs_{n}(x):=\sum_{r=1}^{d} \hat{\beta}(\varphi_{r}) \varphi_{r}(x)
	=
	\frac{1}{n \Delta} \sum_{r=1}^{d} 
	\left[
		\sum_{k=1}^{n} 
		\varphi_{r} \left( X^{(k)}_{\Delta} \right) 
	\right] 
	\varphi_{r} (x),
\end{eqnarray}
is a reasonable estimator for the L{\'e}vy density \(s(x)\). %The estimator \(\hs (x)\) was already considered in \cite{Fig3} for the case Legendre polynomials of orders 0 and 1, and therefore  our research can be viewed as some generalization for more general class of projection estimates.

\section{Main results}
\label{mainres}
In this section, we present our results related to the projection estimator \(\hat{s}(x)\) of the L{\'e}vy density \(s(x)\). First note that by Corollary~8.9 from \cite{Sato}, we immediately get that  \(\hat{\beta}(\varphi_{r})\) defined by \eqref{hb} is a consistent estimator of \(\beta (\varphi_{r})\). Nevertheless, for establishing some theoretical facts, we have to introduce additional assumptions on the rate of this convergence. As in \cite{Fig3}, we assume that the following small-time asymptotic property holds: there exists positive constants \(q\) and \(\Delta^{0}\) such that 
\begin{eqnarray}
\label{supp}
\sup_{x \in D} \left| 
\frac{1}{\Delta} \P \left\{ 
	X_{\Delta} \geq x
\right\} - \nu \left( 
	[x, +\infty)
\right)
\right| < q \Delta, \qquad \forall \; 0<\Delta<\Delta^{0}.
\end{eqnarray}
For instance, this property is fulfilled  when \(s\) is Lipshitz in an open set containing \(D\) and uniformly bounded on \(|x|>M\) for any positive \(M\) (see Proposition~2.1 from \cite{Fig3}).  

In this  paper, we consider the case of high-frequency data with \(T \to \infty\) as \(n \to \infty\). Moreover, the parameter \(m\), which indicates the number of intervals in our construction of the set of basis functions, also tends to infinity with \(n\). With no doubt, the rates of  growth \(m,n,T\) should be somehow coordinated. This can be done in different ways. Mainly for technical reasons, we assume that \(T=n^{\kappa}\) with some \(\kappa>0\) .  Since in high-frequency setup \(\Delta=T/n \to 0\), we get that \(\kappa<1\). The main advantage of such choice is that our further assumptions reduce to only one more restriction on the speed of \(m\), which is presented below by \eqref{lambdan}.

We start the analysis of the distribution of \(\DDD_{n}\) with a technical result related  the ``bias'' of \(\DDD_{n}\), which we  denote below  by \(\Z_{n}\). This result allows to reformulate the problem of finding the asymptotic behaviour of  the distribution function of   \(\Z_{n}\) in terms of Gaussian processes. 
\begin{prop}
\label{FG}
Assume that \eqref{supp} holds. Denote 
\begin{eqnarray*}
	\Z_{n}:=		\sup_{x \in D} \left(	
			\frac{
								\left|
					\hs_{n} (x)  - \E\hs_{n} (x)  
				\right|
			}
			{	
				\sqrt{s(x)}
			}
		\right), \qquad
	F_{n}(u):= \P \left\{
		\Z_{n}  \leq u
	\right\}.
\end{eqnarray*}
Let  \(m\) tend to \(\infty\) as \(n \to \infty\), and moreover \(T = n^{\kappa}\) for some \(\kappa \in (0,1)\), and 
\begin{eqnarray}
\label{lambdan}
	\Lambda_{n}:=m \frac{\sqrt{\log n}}{n^{\kappa/2}} \to 0, \qquad \mbox{as}\quad n \to \infty.
\end{eqnarray}
Then there exist positive constants \(c_{1}, c_{2}, \lambda\) such that
\begin{eqnarray}
\label{Fst1}
		F_{n}\Bigl(u\Bigr) 
		& \leq & 
		\left(
			\breve{F}\left(
				\sqrt{\frac{T}{b-a}} \; u + c_{1} \Lambda_{n}
			\right)
		\right)^{m} +  c_{2} n^{-\lambda},\\
		\label{Fst2}
				F_{n}\Bigl(u\Bigr) 
		& \geq & 
		\left(
			\breve{F}\left(
				\sqrt{\frac{T}{b-a}}\; u - c_{1} \Lambda_{n}
			\right)
		\right)^{m} -  c_{2} n^{-\lambda},			
		\end{eqnarray}
where by \(\breve{F}(\cdot) \) we denote the distribution function  of the random variable 
\begin{eqnarray}
\label{zeta}
	\zeta =\zeta^{J,m} := \sup_{x \in [a, a+\delta)} 
	\left|  \p^{J,m}(x) \right|, \qquad  \p(x)=\p^{J,m}(x):=\sum_{j=0}^{J} Z_{j} \psi_{j}^m (x)
\end{eqnarray}
with i.i.d. standard normal r.v.'s  \(Z_{j},\; j=0, ... , J\). 
\end{prop}
\begin{proof}
	Proof is given in Appendix~A.1 from our preprint~\cite{panov2014a}. 
\end{proof}

Proposition~\ref{FG} allows to  reduce the original problem to the problem of estimating the asymptotic distribution of the supremum of the absolute value of the  Gaussian process \(\p^{J,m}(x):=\sum_{j=0}^{J} Z_{j} \psi_{j}^{m}(x)\)  on the interval \([a, a+\delta)\), where \(x\) plays the role of time.  The next step is 
to infer on the asymptotic behavior of \(\zeta^{J,m}\). 

To the best of our knowledge, there exists no clear theory revealing the asymptotic behaviour of the supremum of any Gaussian process (as well as the asymptotic behaviour of the supremum of absolute values). The  comprehensive overview of this topic is given in \cite{Adler} and \cite{Piterbarg}. Let us mention one interesting result in this direction, the theorem by Marcus and Shepp \cite{MS}, which states that  if a centered Gaussian process \(G_{t}\)  has bounded sample paths with probability 1, then the logarithmic asymptotics is given by
\begin{eqnarray}
\label{Gt}
	\lim_{u \to \infty} \frac{
		\log \P \left\{	
			\sup_{t \in K} G_{t} \geq u
		\right\} 
	}
	{
		u^{2}
	} 
	=
	- \frac{1}{2\sigma_{K}^{2}},
\end{eqnarray}
where \(K\) is a subset of \(\R\) and \(\sigma_{K}^{2} = \sup_{t \in K} \E G_{t}^{2}\). It turns out that the supremum of  \(G_{t}\)   behaves much like a single Gaussian variable \(\xi\) with zero-mean and variance \(\sigma^{2}=\sigma_{K}^{2}\), because the logarithmic asymptotics of  \(\xi\) is the same, i.e., 
\begin{eqnarray*}
	\lim_{u \to \infty} \frac{
		\log \P \left\{	
			\xi \geq u
		\right\} 
	}
	{
		u^{2}
	} 
	=
	- \frac{1}{2\sigma^{2}}.
\end{eqnarray*}
Nevertheless, the main term in the asymptotics of \(\P\{\xi \geq u \}\) is given by 
\begin{eqnarray*}
		\P \left\{	
			\xi \geq u
		\right\} 
	= \frac {\sigma}{u \sqrt{2 \pi}}
	\exp \Bigl(
		- \frac{u^{2}}{2\sigma^{2}}
	\Bigr) \Bigl( 
		1 + o (1)
	\Bigr), \qquad u \to +\infty,
\end{eqnarray*}
which do not necessary coincide with the asymptotics of \(\P\{\sup G_{t} \geq u \}\).

In the next sections, we closely consider the Gaussian processes \(\p^{J,m}(x)\), defined by \eqref{zeta}, where \(\{\psi^{m}_{j}(x)\}\) are the sets of basis functions  listed in Section~\ref{basis}. The next theorem demonstrates the asymptotic behavior of the distribution function of the r.v. 
\begin{eqnarray*}
 \tzeta =\tzeta^{J,m} := \sup_{x \in [a, a+\delta)} 
	 \p^{J,m}(x).\end{eqnarray*}
Later on, we will derive from this theorem the asymptotic behavior of   \(\zeta^{J,m}\).
 \begin{thm}
 \label{main}
 	Let \(u\) grow as \(\delta \to 0\) so that \(\sqrt{\delta}u \to \infty\). Then 
	\begin{eqnarray}
	\longboxed{
	\label{zetaas} 
		\P \left\{ 
			\tzeta^{J,m} \geq u 
		\right\} 
		= \frac{g_{1}(J) }{ \left( \sqrt{\delta} u \right)^{\k}} e^{-g_{2}(J) \cdot \delta u^{2}}
		\left(
			1 + \tau(\sqrt{\delta} u)
		\right),
		}
	\end{eqnarray}
where \(\tau(x) \to 0\) as \(x \to \infty\). Moreover,
\begin{enumerate}[(i)]
\item in case of trigonometric basis,  
\begin{eqnarray*}
	\k =0, \qquad g_{1}(J)  = 
	\left( 
		2 J^{-1}
		 \sum_{j=1}^{J/2} j^{2}
	\right)^{1/2}, \qquad g_{2}( J)  =  (2 J)^{-1};
\end{eqnarray*}
\item in case of Legendre polynomials, 
\begin{eqnarray*}
	\k = 1, \qquad g_{1}( J)  = \sqrt{2} (J+1)/\sqrt{\pi} , \qquad g_{2}(J)  =2^{-1}\left(J+1 \right)^{-2} ;
\end{eqnarray*}
\item in case of wavelets, 
\begin{eqnarray*}
	\k = 1, \qquad g_{1}(J) = 2/ \sqrt{\pi}, 
	\qquad g_{2}(J) =1/4.
\end{eqnarray*}
\end{enumerate}
Furthermore, for some cases the asymptotic behavior of the function \(\tau(x)\) as \(x \to \infty\) is known. In particular, in case (i), \[\tau(x) =\sqrt{J}\left(\sqrt{2\pi}x \cdot g_{1}(J)  \right)^{-1} \left( 1 + o(1) \right);\] in case (iii), \(\tau(x)= O(x^{-2})\). \end{thm}
 \begin{proof}
 As it was mentioned before,  there is no unified approach to find the asymptotics of the distribution of Gaussian process. Since the methodology crucially depends on the properties of covariance function,  we separately prove this result for different basis functions, see Sections \ref{statsec} - \ref{wav}. In the case of trigonometric basis (Section~\ref{statsec}), we efficiently use that the considered Gaussian process is stationary, and apply the Pickands theorem (\cite{Michna}, \cite{Piterbarg}) and some further results on this topic. In the case of Legendre polynomials (Section~\ref{Leg}), we take into account that the variance of the process attains its maximum only in finite number of points, and apply  the double sum method described in \cite{Piterbarg}. Finally, in case of wavelets (Section~\ref{wav}), we directly calculate the asymptotic behaviour of \(\P\{\zeta \geq u \}\).
 \end{proof}
 \begin{rem}
 \label{rr}
In what follows, we will use the following trivial corollary from~\eqref{zetaas}. Let \(u\) grow with \(m\) so that \(u/\sqrt{m} \to \infty\). Then 
 	\begin{eqnarray}
	\label{zetaas2} 
		\P \left\{ 
			\tzeta^{J,m} \geq u 
		\right\} 
		= \frac{h_{1} \; m^{k/2}}{  u^{\k}} 
		\exp\left\{
			-h_{2} \;  u^{2}/ m
		\right\}
			\left(
				1 + \btau\left(
					u / \sqrt{m}
				\right)
			\right), 
	\end{eqnarray} 
	where  \(h_{1} = h_{1}(J):=g_{1}(J) \cdot (b-a)^{-k/2}\),  \(h_{2} =h_{2}(J):= g_{2}(J) \cdot (b-a),\) and \( \btau(x) = \tau \left(\sqrt{b-a} \cdot x\right)\).
 \end{rem}
 Using \eqref{zetaas}, we can derive similar result for the supremum of the absolute value of the Gaussian process.
 \begin{cor} 
 \label{corcor}
 	In the assumptions and notations of the last remark, 
	 \begin{eqnarray}
	\label{zetaas2} 
		\P \left\{ 
			\zeta^{J,m} \geq u 
		\right\} 
		= 2 \frac{h_{1} \; m^{k/2}}{  u^{\k}} 
		\exp\left\{
			-h_{2} \;  u^{2}/ m
		\right\}
			\left(
				1 + \btau\left(
					u / \sqrt{m}
				\right)
			\right). 
	\end{eqnarray} 
 \end{cor}
 \begin{proof}
 The proof can be found in our preprint~\cite{panov2014a}, Appendix~A.2.
\end{proof}
Proposition~\ref{FG} and Theorem~\ref{main} yield the following theorem, which shows the asymptotic distribution of the maximal deviation \(\Z_{n}\).
\begin{thm}
\label{main2}
Let the assumptions on the relation between  \(m, n\)  and \(T\) introduced in Proposition~\ref{FG} be fulfilled. Denote for any \(y \in \R\),
\begin{eqnarray}
\label{un}
	u_{m}:= 
			\frac{y}{
			a_{m}}
			+ \left(
				b_{m}-\frac{c_{m}}{b_{m}}
			\right),
\end{eqnarray}
where 
\begin{eqnarray}
\label{abc}
	a_{m} := 2 h_{2} b_{m}, \quad 
	b_{m}:=
	\sqrt{
\frac{1}{ h_{2}}  \ln\left(
	h_{1} m
\right)}, \quad
	c_{m}:= \frac{k}{2h_{2}} \ln b_{m}.
\end{eqnarray}
Then for any \(y \in \R\),
\begin{eqnarray*}
\longboxed{
 	\P \left\{
		\sqrt{\frac{T}{m}}
		\Z_{n}
		\leq u_{m} 
	\right\} =
	e^{- 2 e^{-y}} \left( 
			1 + R(m)
		\right),}
\end{eqnarray*}
where 
\begin{equation}
\label{s}
R(m) := 
		\btau \left( u_m \right) 
		 - 
		\frac{k}{4 \sqrt{h_2} } 
		\frac{\ln \ln m}{\sqrt{\ln m}} 
		\left( 
			1 + o(1)
		\right) \to 0, \; \mbox{ as } m \to \infty.
\end{equation}
\end{thm}
\begin{proof}
The proof is given in Appendix~\ref{prooffg}.
\end{proof}
%\underline{\textbf{Example.}} 
In the next theorem, we get the asymptotic distribution of \(\DDD_{n}\) given by \eqref{DDD} from the asymptotic distribution of its ``bias'' \(\Z_{n}\) defined in Proposition~\ref{FG} and later analyzed in Theorem \ref{main2}.
\begin{thm}
\label{main3}
	In the assumptions and notations of Theorem~\ref{main2},
	\begin{eqnarray*}
\boxed{
 	\P \left\{
		\sqrt{\frac{T}{m}}
		\Z_{n}
		\leq u_{m} + \cc n^{3 \kappa/2-1} m^{1/2}
	\right\} =
	e^{-2 e^{-y}} \left( 
			1 + R(m)
		\right)
		}
\end{eqnarray*}
with \(\cc=q J \C_1 (\C_1 + \C_2) /(b-a)\), where \(\C_{1}, \C_{2}\) are defined by \eqref{cond1}, and \(q\) is defined in \eqref{supp}.
\end{thm} 
\begin{proof}
	The proof is given in Appendix~\ref{acorcor}.
\end{proof} 
Note that here we have the usual trade-off between the deviation \((\hs_{n}(x) - \E\hs_{n}(x))\) and the bias \((\E\hs_{n}(x) - s(x))\). In fact, while \(u_m\) (which is `''responsible'' for the deviation) decays with \(m\), the second term \(\cc n^{3 \kappa/2-1} m^{1/2}\) (which is `''responsible'' for the bias) grows.  In this respect, the optimal choice of \(m\) is \(n^{2-3 \kappa}\), which is possible under the assumption \(\kappa \in (4/7, 2/3)\).

According to Theorem~\ref{main}, the function \(\btau(x) = \tau \left(\sqrt{b-a} \cdot x\right)\) is known  for some sets of basis functions. For instance, in case of trigonometric polynomials, \[\tau(u_{m})= C u_{m}^{-1} (1+o(1)) = C  \frac{\sqrt{h_{2}}}{\sqrt{\ln m}}(1 + o(1)),\] where the constant \(C\) can be explicitly computed from Theorem~\ref{main}. Since \(k=0\), we get 
\begin{equation*}
R(m) =\left(	
\frac{J h_{2}}{ 2\pi (b-a) h_{1}^{2}}
\right)^{1/2} \cdot \frac{1}{\sqrt{\ln m}}(1 + o(1)), \qquad m \to \infty.
\end{equation*} 
In case of wavelets, \(\tau(u_{m})=O(u_{m}^{-2}) = O(1/\ln m)\) and \(k=1\),  and therefore 
\begin{equation*}
R(m) = - 
		\frac{1}{4 \sqrt{h_2} } 
		\frac{\ln \ln m}{\sqrt{\ln m}} 
		\left( 
			1 + o(1)
		\right), \qquad m \to \infty.
\end{equation*}
Therefore, the rates of convergence are typically of logarithmic order. Nevertheless, at least in the case of trigonometric basis, we can also find a sequence of accompanying laws, which approximate the distribution of \(\DDD_{n}\) with polynomial rate. The next theorem clarifies this point. 
\begin{thm} \label{thm35}  
Consider the case of trigonometric basis and let the assumptions of 
Theorem~\ref{main2} be fulfilled, and moreover \(J \geq b-a\). Define the sequence of distribution functions
\begin{equation*}
A_{m}(y):=\begin{cases}
\exp \left\{ -2\exp\left\{-y-\frac{y^{2}}{ 4 \ln(h_{1} m)} \right\}
- 2 m\left(
1-\Phi \left(
	u_{m} \sqrt{\frac{b-a}{J}}
\right)
\right)\right\}, \\&\hspace{-2.2cm}\mbox{if }y\geq -b_{m}^{3/2}, \\ 
0,  &\hspace{-2.2cm}\mbox{if }y<-b_{m}^{3/2},%
\end{cases}   
\end{equation*}
where \(u_{m}\) and \(b_{m}\) are defined by \eqref{un}-\eqref{abc}, and \(\Phi(\cdot)\) is the distribution function of the standard normal distribution. Denote 
\[
A_{m}^{-} (y) = A_{m}\left(
		y-\bar{c} n^{3 \kappa/2-1} m^{1/2}b_{m}
	\right), \quad 
A_{m}^{+} (y) = A_{m}\left(
		y + \bar{c} n^{3 \kappa/2-1} m^{1/2}b_{m}
	\right),
\]
where \(\bar{c}:= 2 h_{2}\cc\), and \(\cc\) was defined in Theorem~\ref{main3}.
Then for sufficiently large \(m\) and for any \(y\in \R\),
\begin{eqnarray}
\label{amm}
\boxed{
	A_{m}^{-}(y) 
	- \frac{C}{m ^{\beta}}
	\leq
 	\P \left\{
		\sqrt{\frac{T}{m}}
\Z_{n}
		\leq u_{m}
	\right\}
	\leq 
	A_{m}^{+}(y) 
	+\frac{C}{m^{\beta}}
}
\end{eqnarray}
with some positive constants \(C\) and \(\beta\).
\end{thm}
\begin{proof}
Proof is given in Section~\ref{thm335}. 
\end{proof}%Note that the inequalities in \eqref{amm} are  close to the notion of L{\'e}vy metric. 
Note that the statement of the theorem is formally true with any choice of \(n, m\) satisfying the assumptions, but the main interest is drawn to the case 
\[
 n^{3 \kappa/2-1} m^{1/2}b_{m}\to 0, \qquad n,m \to \infty,
 \]
which occurs for instance when \(\kappa \in (0, 4/7)\), or when \(m=n^{\alpha}\) with 
\(0<\alpha<\min(2-3\kappa, \kappa/2)\) and \(\kappa \in (0, 2/3)\). In the later situation, \eqref{amm} can be reformulated in terms of L{\'e}vy distance, which is defined for any two distributions functions \(G_{1}\) and \(G_{2}\) by
\[
	L(G_{1}, G_{2}) = \inf \Bigl\{
		\eps>0: \; G_{1}(x-\eps) - \eps \leq G_{2}(x) \leq
		G_{1} (x+\eps) + \eps, \quad \forall x \in \R
	\Bigr\}.
\]
In fact,   the L{\'e}vy distance between \(A_{m}(y)\) and the distribution function 
\begin{eqnarray*}
	G_{m}(y) &:=& \P \left\{
		2 h_{2} b_{m} \left[
		\sqrt{\frac{T}{m}}
		\sup_{x \in D} \left(	
			\frac{
				\left|
					\hs_n (x)  -\E\hs_n (x)
				\right|
			}
			{	
				\sqrt{s(x)}
			}
		\right)
		-b_{m}
	\right]
	\leq y 
	\right\}\\ &=& 
	 	\P \left\{
		\sqrt{\frac{T}{m}}
		\sup_{x \in D} \left(	
			\frac{
				\left|
					\hs_n (x)  - \E\hs_n (x)  
				\right|
			}
			{	
				\sqrt{s(x)}
			}
		\right)
		\leq u_{m}
	\right\},
\end{eqnarray*}
is bounded for large \(m\) by 
\[
\boxed{
	L\left(
		A_{m}, G_{m}
	\right) 
	\leq \widetilde{C} / m^{\tilde{\beta}}, }
	 \qquad \widetilde{C}>0, \quad
	 \tilde{\beta} = \min\left(\beta, \left(1 - \frac{3\kappa+ \alpha}{2}\right)\frac{1}{\alpha}\right).
\]
So, Theorem~\ref{main3} yields that convergence to the Gumbel distribution is quite slow, and therefore we cannot state that for some realistic \(m\) the maximal deviation distribution is close to its asymptotic distribution. Such situations are typical for similar types of problems, see, e.g.,  \cite{Hall} and \cite{KP}. Nevertheless, from Theorem~\ref{thm35}, we get that the distance between maximal deviation distribution and  the distribution function \(A_{m}(y)\) converges to zero at polynomial rate.

 Note that the rate of convergence formally becomes  logarithmically slow if we expand
additionally the double exponent in the definition of \(A_{m}(y)\), that is,
\begin{equation}
\label{unverbesser}
\sup_{y \in \R}
\left|
 	\P \left\{
		\sqrt{\frac{T}{m}}
		\Z_{n}
		\leq u_{m}
	\right\}
	-
	e^{-2 e^{-y}}\right| = \frac{C}{\sqrt{\ln m}} (1+o(1))
\end{equation}
for some positive constant  $C$ and \(m \to \infty\). We clarify this point in Section~\ref{Taylor}.

In the next 3 sections, we separately consider different choices of basis functions: trigonometric basis, Legendre polynomials and wavelets.
\section{Stationary case}
\label{statsec}
\subsection{Trigonometric basis}
Let us consider the case of trigonometric basis, 
		\begin{eqnarray*}
\Bigl\{ 
	\psi_{j}^{m}(x), \; j=0..J
\Bigr\} 
&=&
\Bigl\{  
	\chi_{0}(x) = \frac{1}{\sqrt{\delta}}, \qquad\chi_{j}(x) = \sqrt{\frac{2}{\delta}} \cos\left( 
		2 j \pi (x-a) / \delta
	\right),
	 \\
&& \hspace{0.4cm}
\tilde\chi_{j}(x) = \sqrt{\frac{2}{\delta}}\sin\left( 
		2 j \pi (x-a) / \delta
	\right),
\qquad j=1 .. (J/2)
\Bigr\}
\end{eqnarray*}
on the interval \([a, a+\delta)\). Changing the variables in \eqref{zeta}: \(x \to \tau=(x-a)/ \delta\), we get
\begin{eqnarray*}
	\tzeta &=& \max_{\tau \in [0, 1]} \p (\tau),\\
	\p(\tau) &:=&
			\frac{Z_{0}}{\sqrt{\delta}}
			+ 
			\sqrt{
				\frac{2}{\delta}
			}
			\sum_{j=1}^{J/2}
			\left[
				Z_{j}
				\cos\left(
					2 \pi j \tau
				\right)
				+
				\tilde{Z}_{j}
				\sin\left(
					2 \pi j \tau
				\right)
			\right], \qquad \tau \in [0,1],
\end{eqnarray*}
where all \(Z_{j}\) and \(\tilde{Z}_{j}\) are i.i.d. standard normal r.v.'s.
Note that \(\p (\tau)\)  is a stationary Gaussian process with covariance function 
\begin{eqnarray*}
	r(\tau) = 
	\frac{1}{\delta}  + 
	\frac{2}{\delta}
	\sum_{j=1}^{J/2} 
	\cos\left(
		2 \pi j \tau 
	\right).
\end{eqnarray*}
It is a worth mentioning that the process \(\p(\tau)\) has the same mean (equal to zero) and covariance function as the process \( \sqrt{ J /\delta} \;\widetilde{\p}\left(2\pi\tau\right)\), where 
\begin{eqnarray*}
		\widetilde{\p} (\tau) &:=&
				\frac{1}{\sqrt{J}}
	\left\{
			Z_{0}
			+ 
			\sqrt{ 2 }
			\sum_{j=1}^{J/2}
			\left[
				Z_{j}
				\cos\left(
					 j  \tau 
				\right)
				+
				\tilde{Z}_{j}
				\sin\left(
					j \tau 
				\right)
			\right]
	\right\},
\end{eqnarray*}
and therefore the r.v. \(\tzeta\) has the same distribution as 
\[
	 \tzeta \eqd \sqrt{\frac{J}{\delta} }\cdot \sup_{\tau \in [0, 2 \pi]} \wp (\tau).
\]
Note that the  process \(\wp(\tau)\) doesn't depend on \(n\). 

\subsection{The Pickands theorem}
In this section, we prove Theorem~\ref{main} (i) by using the Pickands theorems. This theorem stands that for any continuous stationary Gaussian process \(X_{t}, \; t \in [0,M]\) with zero mean and covariance function satisfying the assumptions
\begin{eqnarray*}
	r(t) &=& \E\left[
		X_{t+s} X_{t}
	\right]
	=
	1-|t|^{\alpha} +o(|t|^{\alpha}), \qquad t \to 0,\\
	r(t)& <& 1, \qquad \forall t>0,
\end{eqnarray*}
the asymptotic behaviour of the probability \(\P\left\{ \sup_{t \in [0,M]}X_{t}>u\right\}\) as \(u \to \infty\) is given by
\begin{eqnarray*}
	\P\left\{ \sup_{t \in [0,M]}X_{t}>u\right\} = 
	H_{\alpha} M u^{2/\alpha} \left(
		1 - \Phi(u)
	\right)
	\left( 
		1  + o(1)
	\right), \qquad u \to \infty.
\end{eqnarray*}
The next result shows the asymptotics of  the distribution function of \(\tzeta\).
\begin{lem}  	Let \(u\) grow as \(\delta \to 0\) so that \(\sqrt{\delta}u \to \infty\). Then 
\label{stat}
\begin{eqnarray}
	\P \left\{
		 \tzeta
		 \geq u
	\right\} 
	&=&
	C \; \sqrt{\delta} \; u \; \left( 
	1 -	\Phi\left(
			u \sqrt{\delta}/\sqrt{J}
		\right)
	\right)
	 \; \left(
		1 + o(1)
	\right),
\label{zzeta}
\end{eqnarray}
where the constant \(C\) is equal to 
\[
	C:=\frac{2}{J}
	\sqrt{
		\pi \sum_{j=1}^{J/2} j^{2}
	}.
\]
\end{lem} 
\begin{rem}
\label{remmich}
Since for \(u \to \infty\), 
\begin{eqnarray}
\label{aphi}
	1 - \Phi(u) = \frac{1}{\sqrt{2 \pi} u} e^{-u^{2}/2}
	\left(
		1 - \frac{1}{u^2} + 
			o\left( 
				\frac{1}{u^2}
			\right)
		\right)
\end{eqnarray}
(see, e.g., \cite{Michna}), we get that \eqref{zzeta} is equivalent to
\begin{eqnarray*}
	\P \left\{
		 \tzeta
		 \geq u
	\right\} 
	&=&
	\frac{C \sqrt{J}}{\sqrt{2 \pi}}\; 
	\exp\{
		- u^{2} \delta / ( 2 J)
	\}
	 \; \left(
		1 + o(1)
	\right), \qquad u \to \infty,
\end{eqnarray*}
and therefore Lemma~\ref{stat} is equivalent to Theorem~\ref{main} (i).
\end{rem}

\begin{proof}To prove this result,  we apply the Pickands theorem, see \cite{Michna} or \cite{Piterbarg}, to the process 
\begin{eqnarray*}
		\wp_{c} (\tau) &:=&
				\frac{1}{\sqrt{J}}
	\left\{
			Z_{0}
			+ 
			\sqrt{ 2 } 
			\sum_{j=1}^{J/2}
			\left[
				Z_{j}
				\cos\left(
					 j  \tau / \sqrt{c}
				\right)
				+
				\tilde{Z}_{j}
				\sin\left(
					j \tau / \sqrt{c}
				\right)
			\right]
	\right\},
\end{eqnarray*}
where positive constant \(c\) will be defined later. Note that 
\[
\tzeta \eqd \sqrt{\frac{J}{\delta}} \cdot \sup_{\tau \in [0, 2 \pi \sqrt{c}]} \wp_{c} (\tau).
\]
The process \(\wp_{c}(\tau)\) is a continuous stationary centered Gaussian process with covariance function 
\begin{eqnarray*}
	\tilde{r}(\tau) = 
	\frac{1}{J}  + 
	\frac{2}{J}
	\sum_{j=1}^{J/2} 
	\cos\left(
		j \tau / \sqrt{c}
	\right).
\end{eqnarray*}
Choosing \(c= J^{-1}\sum_{j=1}^{J/2} j^{2}\), the covariance function \(r(\tau)\) allows the following decomposition near \(0\):
\begin{eqnarray*}
	r(\tau) = 1- \tau^{2} + o(\tau^{2}), \qquad \tau \to 0.
\end{eqnarray*}
It is worth mentioning that the last condition of the Pickands theorem, which states that \(r(\tau)  <1\) for all \(\tau>0\), doesn't hold in our situation. To avoid this difficulty, we apply the following approach, which was suggested  by Vladimir Piterbarg in private communication. The authors greatfully acknowledge his help.

The proposed approach is based on the observation that  the Pickands theorem is applicable on any interval \([0, t]\) where \(t\) is strictly smaller than \(2\pi \sqrt{c}\).  In fact, the covariance function of the process \(\wp(\tau)\) can be bounded by
\begin{eqnarray*}
	\tilde{r}(\tau)  \leq 
	\frac{1}{J} \Bigl(	
		1 + 2  J/2
	\Bigr) =1,
\end{eqnarray*}
and the last inequality is in fact an equality if and only if \(\tau\) is proportional to \(2\pi \sqrt{c}\). With this idea in mind, we define for \(q \in [0,2 \pi]\)
\begin{eqnarray*}
	M(q, u) := \P\left\{	
		\sup_{\tau \in [0, q \sqrt{c}]} \wp_{c}(\tau) > u 
	\right\} 
\end{eqnarray*}
In this notation, we are interested in the asymptotics of \(M(2\pi,u )\) as \(u \to \infty\). Fix some small \(\eps>0\) and  note that 
\begin{eqnarray*}
	 M(2 \pi-\eps, u) \leq M(2\pi,u ) \leq M (2\pi - \eps,u) + M(\eps,u).
\end{eqnarray*}
Next, we divide both parts of both inequalities by  \(M (2\pi - \eps, u)\) and apply the Pickands theorem, 
\begin{eqnarray*}
	\lim_{u \to \infty} \frac{ M(\eps,u)}{M (2\pi - \eps,u)} = 
	\frac{\eps}{2\pi - \eps} < 0.2 \eps.
\end{eqnarray*}
Since \(\eps\) was chosen arbitrary, we get 
\begin{eqnarray*}
	M(2 \pi, u) = H_{2} 2\pi\sqrt{c} u (1- \Phi(u)) (1+o(1)),
\end{eqnarray*}
where \(H_{2}=1/\sqrt{\pi}\) is the Pickands constant, see \cite{DK}. This observation completes the proof.
\end{proof}

\underline{\textbf{Example.}} 
Consider the case \(J=3\).  Let us analyze the distribution of the random variable
\begin{eqnarray*}
\sqrt{\frac{3}{\delta}} \cdot \sup_{\tau \in [0, 2 \pi]} 
		\wp (\tau) 
		&=&
		\frac{\sqrt{3}}{\sqrt{\delta}}
		\cdot
		\sup_{\tau \in [0, 2 \pi]} 
		\left\{
			\frac{Z_{0}}{\sqrt{3}} 
			+ 
			\sqrt{\frac{2}{3}}
			\left[
				Z_1
				\cos \tau 
				+
				Z_2
				\sin \tau 
			\right]
	\right\}\\
	&=&
			\frac{\sqrt{3}}{\sqrt{\delta}}
		\cdot
		\left\{
			\frac{Z_{0}}{\sqrt{3}} 
			+ 
			\sqrt{\frac{2}{3}}
			\sup_{\tau \in [0, 2 \pi]} 
			\left[
				Z_1
				\cos \tau 
				+
				Z_2
				\sin \tau 
			\right]
	\right\}.
\end{eqnarray*}
Note that the expression under supremum can be interpreted as the projection of the random vector \(\vec{Z} := (Z_{1}, Z_{2})\) on the direction \((\cos \tau, \sin \tau)\). Since the supremum is taken over all possible directions, we get that this supremum is equal to the length of the vector \(\vec{Z}\), which is distributed as \(\sqrt{\xi}\), where \(\xi\) has a \(\chi\)-squared distribution with 2 degrees of freedom.
Therefore 
\begin{eqnarray}
\label{int1}
\P \left\{
		 \sqrt{\frac{3}{\delta}} \cdot \sup_{\tau \in [0, 2 \pi]} 
		\wp (\tau) 
		 \geq u
	\right\} 
	=
	\int_{u \sqrt{ \delta} / \sqrt{3}}^{\infty}
	p(x) dx,
\end{eqnarray}
where \(p(u)\) is the convolution of the density of the r.v. \(\sqrt{1/3} Z_{0}\), which is equal to \(\sqrt{3/(2\pi)} \exp\{ -3 x^{2} /2\}\), and the denstiy of the r.v. \(\sqrt{\xi}\), which is equal to \((3/2) x \exp\{ -3x^{2}/4\} I\{x>0\}\). This density can be explicitly computed, 
\begin{eqnarray*}
	p(x) = \sqrt{\frac{2}{3}}  x e^{-x^{2}/2} \left(
		1 - \Phi(-\sqrt{2} x)
	\right)
	+ 
	\frac{1}{\sqrt{6\pi}} e^{-3x^{2}/2}.
\end{eqnarray*} The main term in the asymptotics of the integral in \eqref{int1} is given by 
\begin{eqnarray*}
	\sqrt{\frac{2}{3}}
	\int_{u\sqrt{\delta} / \sqrt{3} }^{\infty} x \exp\{-x^{2}/2\} dx = 
	\sqrt{\frac{2}{3}}
	\exp\{-u^{2} \delta / 6\},
\end{eqnarray*}
which in fact coincides with the representation from Remark~\ref{remmich}.

\subsection{Asymptotic behavior of  $\tau(x)$ as $x \to \infty$} 
 To find the second term in the asymptotics of \(\P\left\{ \tzeta \geq u \right\}\), we apply  techniques from  \cite{Piterbarg}, Section~F, to the process \(\wp_{c} (\tau)\). These techniques are based on the assumption that the following determinant  is not equal to zero:
\begin{equation*}
\DD :=\left\vert 
\begin{tabular}{llll}
$1$ & $0$ & $\left\langle S(0),S(t)\right\rangle $ & $\left\langle
S(0),S^{\prime }(t)\right\rangle $ \\ 
$0$ & $1$ & $\left\langle S^{\prime }(0),S(t)\right\rangle $ & $\left\langle
S^{\prime }(0),S^{\prime }(t)\right\rangle $ \\ 
$\left\langle S(0),S(t)\right\rangle $ & $\left\langle S(0),S^{\prime
}(t)\right\rangle $ & $1$ & $0$ \\ 
$\left\langle S^{\prime }(0),S(t)\right\rangle $ & $\left\langle S^{\prime
}(0),S^{\prime }(t)\right\rangle $ & $0$ & $1$%
\end{tabular}%
\right\vert, 
\end{equation*}%
where 
\[
S(t)=\sqrt{\frac{2}{J}}\left( \frac{1}{2},\cos \left( \frac{t}{\sqrt{c}}%
\right) ,\sin \left( \frac{t}{\sqrt{c}}\right) ,...,\cos \left( \frac{Nt}{%
\sqrt{c}}\right) ,\sin \left( \frac{Nt}{\sqrt{c}}\right) \right) ,
\]%
\[
S(0)=\frac{1}{\sqrt{J}}\left( 1,\sqrt{2},0,...,\sqrt{2},0\right) ,
\]%
\[
S^{\prime }(t)=\sqrt{\frac{2}{Jc}}\left( 0,-\sin \left( \frac{t}{\sqrt{c}}%
\right) ,\cos \left( \frac{t}{\sqrt{c}}\right) ,...,-N\sin \left( \frac{Nt}{%
\sqrt{c}}\right) ,N\cos \left( \frac{Nt}{\sqrt{c}}\right) \right) ,
\]%
\[
S^{\prime }(0)=\sqrt{\frac{2}{Jc}}\left( 0,0,1,0,2,...,0,N\right),
\]%
and \(N=J/2\). Denote 
\[
B=\left( 
\begin{tabular}{ll}
$\left\langle S(0),S(t)\right\rangle $ & $\left\langle S(0),S^{\prime
}(t)\right\rangle $ \\ 
$\left\langle S^{\prime }(0),S(t)\right\rangle $ & $\left\langle S^{\prime
}(0),S^{\prime }(t)\right\rangle $%
\end{tabular}%
\right) ,
\]%
where%
\begin{eqnarray*}
\left\langle S(0),S(t)\right\rangle &=& \frac{1}{J}\left( 1+2\sum_{k=1}^{N}\cos
\left( \frac{kt}{\sqrt{c}}\right) \right),\\
\left\langle S(0),S^{\prime
}(t)\right\rangle &=&-\frac{2}{J\sqrt{c}}\sum_{k=1}^{N}k\sin \left( \frac{kt}{%
\sqrt{c}}\right) ,  \\
\left\langle S^{\prime }(0),S(t)\right\rangle &=&\frac{2}{J\sqrt{c}}%
\sum_{k=1}^{N}k\sin \left( \frac{kt}{\sqrt{c}}\right) =-\left\langle
S(0),S^{\prime }(t)\right\rangle,\\
\left\langle S^{\prime }(0),S^{\prime
}(t)\right\rangle &=&\frac{2}{Jc}\sum_{k=1}^{N}k^{2}\cos \left( \frac{kt}{%
\sqrt{c}}\right).
\end{eqnarray*}
Note that%
\begin{equation}
-\frac{2N-1}{2N+1}=\frac{1}{J}\left( 1-2N\right) \leq \left\langle
S(0),S(t)\right\rangle \leq \frac{1}{J}\left( 1+2N\right) =1,  \label{a4}
\end{equation}%
\begin{equation}
-1=-\frac{2}{Jc}\sum_{k=1}^{N}k^{2}\leq \left\langle S^{\prime
}(0),S^{\prime }(t)\right\rangle \leq \frac{2}{Jc}\sum_{k=1}^{N}k^{2}=1
\label{a5}
\end{equation}%
The r.h.s. equality in (\ref{a4}) and (\ref{a5}) is attained only if $t/
\sqrt{c}=\left\{ 0\},\{2\pi \right\} .$The determinant $\DD $ has a block structure%
\[
\DD =\left\vert 
\begin{tabular}{ll}
$I$ & $B$ \\ 
$B$ & $I$%
\end{tabular}%
\right\vert,
\]
and therefore%
\begin{equation}
\DD =\left\vert I\right\vert \cdot \left\vert I-BI^{-1}B\right\vert
=\left\vert I -  B^{2}\right\vert .  \label{6}
\end{equation}%
Note that  $\DD \neq 0$ if and only if the number one is not an eigenvalue of
the matrix $B^{2}.$ It is well known that the eigenvalues of $B^{2}$ are
equal (with multiplicity) to the squares of the eigenvalues of the matrix $B
$ (\cite{Pro}, exercise 1077, p.145). Note that the direct computation of the eigenvalues is not quite trivial because  $B$ is not symmetric. So we
prove that $\DD \neq 0$ if we show that $\pm 1$ cannot be the eigenvalues
of the matrix $B.$ Suppose that $\lambda _{1}=1$ is the eigenvalue of $B$
and let $x=\left( a,b\right) $ be the corresponding eigenvector. Then we have%
\[
\left( 
\begin{tabular}{ll}
$\left\langle S(0),S(t)\right\rangle $ & $\left\langle S(0),S^{\prime
}(t)\right\rangle $ \\ 
$\left\langle S^{\prime }(0),S(t)\right\rangle $ & $\left\langle S^{\prime
}(0),S^{\prime }(t)\right\rangle $%
\end{tabular}%
\right) \left( 
\begin{tabular}{l}
$a$ \\ 
$b$%
\end{tabular}%
\right) =\left( 
\begin{tabular}{l}
$a$ \\ 
$b$%
\end{tabular}%
\right) ,
\]%
\begin{equation}
a\cdot \left\langle S(0),S(t)\right\rangle +b\cdot \left\langle
S(0),S^{\prime }(t)\right\rangle =a  \nonumber
\end{equation}%
\[
a\cdot \left\langle S^{\prime }(0),S(t)\right\rangle +b\cdot \left\langle
S^{\prime }(0),S^{\prime }(t)\right\rangle =b
\]%
Suppose that $\ a\neq 0$ (the case $b\neq 0$ can be consided analogously).
Denoting $d=b / a$ we obtain 
\begin{equation}
\left\langle S(0),S(t)\right\rangle +d\cdot \left\langle S(0),S^{\prime
}(t)\right\rangle =1  \label{a7}
\end{equation}%
\begin{equation}
-\left\langle S(0),S^{\prime }(t)\right\rangle +d\cdot \left\langle
S^{\prime }(0),S^{\prime }(t)\right\rangle =d  \label{a8}
\end{equation}%
Substituting $\left\langle S(0),S^{\prime }(t)\right\rangle $ from (\ref{a8}) into (\ref{a7}), we obtain

\[
\left\langle S(0),S(t)\right\rangle +d^{2}\cdot \left( \left\langle
S^{\prime }(0),S^{\prime }(t)\right\rangle -1\right) =1,
\]%
\begin{equation}
\left[ \left\langle S(0),S(t)\right\rangle -1\right] +d^{2}\cdot \left[
\left\langle S^{\prime }(0),S^{\prime }(t)\right\rangle -1\right] =0
\label{a9}
\end{equation}%
It follows from (\ref{a4}) and (\ref{a5}) that both terms in square brackets
in (\ref{a9}) are non positive and, hence this equality is possible if and
only if $\left\langle S(0),S(t)\right\rangle =1$ and this is possible if and
only if $t /\sqrt{c}=\left\{ 0\},\{2\pi \right\}.$ Therefore, \(\DD \ne 0\) and we are able to apply (F.3) from \cite{Piterbarg}, 
\begin{eqnarray}
\label{Piter}
\P\left\{
	\sup_{\tau \in [0, 2\pi \sqrt{c}]}\wp_{c} (\tau)  > u
\right\} 
=
\sqrt{2 c} e^{-u^{2}/2} + \left(
	1 - \Phi(u) 
\right) 
+ \rho(u),  
\end{eqnarray}
where \(|\rho(u)| \lesssim e^{-u^{2}(1+\chi)/2}, \; u \to \infty\) for some \(\chi>0\). Hence,  
\begin{eqnarray}
\nonumber
\P\left\{ 
\tzeta>u
\right\} 
&=&
\P\left\{ 
	\sqrt{\frac{J}{\delta}}
	\sup_{\tau \in [0, 2\pi \sqrt{c}]}\wp_{c} (\tau)
	> u
\right\}\\
\label{rhorho}
&=&
\sqrt{2 c} e^{-\delta u^{2}/2J} + \left(
	1 - \Phi(u \sqrt{\delta/J}) 
\right) 
+ \rho(u \sqrt{\delta/J}).
\end{eqnarray}
Note that the statement of the Theorem~\ref{main} (i) immediately follows from \eqref{rhorho}, and therefore the last lines of reasoning can be viewed as another  proof of this theorem. Moreover, from \eqref{rhorho} we get  the exact form of the remainder term \(\tau(x)\), because by \eqref{aphi},
\[ 
\tau(x):= \frac{	1 - \Phi(x /\sqrt{J}) }{\sqrt{2 c} e^{-x^{2}/2J}}  \asymp \frac{\sqrt{J}}{2 \sqrt{\pi c} x}, \qquad x \to \infty.
\]
This observation completes the proof. \newline

\subsection{Sequence of accompanying laws}
\label{thm335}
\underline{\textbf{Proof of Theorem~\ref{thm35}.}} 

The main idea of the proof is to show that the ``bias'' satisfies 
\begin{eqnarray*}
\sup_{y \in \R}
\left|
 	\P \left\{
		\sqrt{\frac{T}{m}}
		\sup_{x \in D} \left(	
			\frac{
				\left|
					\hs_n (x)  - \E \hs_n (x)  
				\right|
			}
			{	
				\sqrt{s(x)}
			}
		\right)
		\leq u_{m}
	\right\}
	-
	A_{m}(y) 
\right| \leq C m ^{-\beta},
\end{eqnarray*}
(see steps 1 and 2) and afterwards to study the entire deviation \(\DDD_{n}\) (step 3). \newline

\textbf{1.}  From the proof of Theorem~\ref{main2}  (see Appendix~\ref{proofmain2}), we get that 
\[\P\left\{
		\sqrt{\frac{T}{m}}
		\sup_{x \in D} \left(	
			\frac{
				\left|
					\hs_n (x)  - \E\hs_n (x)  
				\right|
			}
			{	
				\sqrt{s(x)}
			}
		\right)
		\leq b_{m} + \frac{y}{2 h_{2}b_{m}}
	\right\}
=
	e^{W_{m}} e^{o(W_{m})}+c_{2} n^{-\lambda},
\]
where 
\[
	W_{m} = -m \P\left\{
		\zeta^{J,m} \geq \sqrt{m} u_{m}
	\right\},
	\qquad 
	u_{m}:=	b_{m} + \frac{y}{2 h_{2} b_{m}}.
\]
Using \eqref{rhorho} and notations of Remark~\ref{rr}, we get
 \[  W_{m}
	=
	- 2 h_{1} m \exp\left\{-h_{2}  u_{m}^{2}\right\}-
	2 m\left(
		1-\Phi \left(u_{m} \sqrt{\frac{b-a}{J}}\right)
	\right)+R,
\]
where%
\begin{equation}
\left\vert R\right\vert \leq m \cdot \exp\left\{-\frac{(1+\chi )}{2} \sqrt{\frac{b-a}{J}}u_{m}^{2}\right\}.  \label{ee2}
\end{equation}%
By the definition of $u_{m}$,
\begin{equation}
h_{1} m \exp\left\{-h_{2}  u_{m}^{2}\right\} =
\exp \left\{ 
	- y- \frac{y^{2}}{4 \ln (h_{1} m)}
\right\}.
\end{equation}%
Note that for $y\geq - b_{m}^{3/2}$, for any $\varepsilon >0$ and sufficiently large $m$%
\begin{equation}
u_{m} = b_{m} + \frac{y}{2 h_{2} b_{m}}\geq b_{m} - \frac{1}{2 h_{2}}\sqrt{b_{m}}\geq (1-\varepsilon )b_{m}.
\label{31}
\end{equation}%
Hence, continuing the estimation of \(|R|\)  in \eqref{ee2}, we get that  for sufficiently small $\varepsilon >0$ and sufficiently large $m$ 
\begin{eqnarray*}
\left\vert R\right\vert &\leq& m \exp\left\{-\frac{(1+\chi )}{2} \sqrt{\frac{b-a}{J}}
\left(
	1- \eps
\right)^{2} b_{m}^{2}\right\}\\
&\lesssim&
 m \exp\left\{-(1+\chi ) \left(
	1- \eps
\right)^{2} \sqrt{\frac{J}{b-a}}
\ln (m)\right\}
\leq
C_{1}m^{-\beta},\qquad\beta>0, 
\end{eqnarray*}
because \(b_{m}^{2} \asymp \ln(m)/h_{2} =  2J \ln(m) /(b-a)\) by the definition of \(b_{m}\), \(J>b-a\) by our  assumption, and \(\eps\) can be chosen arbitrarly small.
Therefore,   we obtain that 
\begin{multline*}
\sup_{y\geq
-b_{m}^{3/2}}
\left|
 	\P \left\{
		\sqrt{\frac{T}{m}}
		\sup_{x \in D} \left(	
			\frac{
				\left|
					\hs_n (x)  - \E\hs_n (x)  
				\right|
			}
			{	
				\sqrt{s(x)}
			}
		\right)
		\leq b_{m} + \frac{y}{2 h_{2 }b_{m}}
	\right\}
	-
	A_{m}(y)
\right| \\
= \sup_{y\geq
-b_{m}^{3/2}}\left| 
	A_{m}(y)
\right| \cdot
\left|
	e^{R} - 1
\right|
\asymp \left| R \right|
\leq C_{1} m ^{-\beta}.
\end{multline*} 

\textbf{2.} For $y<-b_{m}^{3/2}$,
we obtain 
\begin{eqnarray*}
W_{m}
&=&
	-m \P\left\{
		\zeta^{J,m} \geq
		\sqrt{m} \left(  
			b_{m} + \frac{y}{2 h_{2} b_{m}}
		\right)
	\right\}\\
&\leq &	
-mP\left\{
\zeta^{J,m} >
\sqrt{m} \left(  
	b_{m} - \frac{1}{2 h_{2}}\sqrt{b_{m}}
\right)
\right\}\\
&= &
-2 h_{1} m \exp\left\{ 
	-h_{2} b_{m}^{2} + b_{m}^{3/2} - \frac{b_{m}}{4 h_{2}}\right\}\\
	&\lesssim&
	- 2 \exp \left\{
		 b_{m}^{3/2}
	\right\}
	\asymp
	-2  \exp \left\{
		\frac{ 
		 	(\ln m)^{3/4}
		}{
			 h_{2}^{3/4}
		}
	\right\}.
\end{eqnarray*}
Therefore,
\begin{multline*}
\sup_{y<
-b_{m}^{3/2}}
\left|
 	\P \left\{
		\sqrt{\frac{T}{m}}
		\sup_{x \in D} \left(	
			\frac{
				\left|
					\hs_n (x)  - \E\hs_n (x)  
				\right|
			}
			{	
				\sqrt{s(x)}
			}
		\right)
		\leq b_{m} + \frac{y}{2 h_{2 }b_{m}}
	\right\}
	-
	A_{m}(y)
\right| \\\leq
\exp\left\{
		- 2 \exp \left\{
		\frac{ 
		 	(\ln m)^{3/4}
		}{
			 h_{2}^{3/4}
		}
	\right\}
\right\} \lesssim m^{-K}
\end{multline*} 
for any positive \(K\). 

\textbf{3.} By \eqref{bias} and the first inequality in \eqref{f64}, we know that 
\begin{eqnarray*}
		\sqrt{\frac{T}{m}} \DDD_{n} \leq 
	\sqrt{\frac{T}{m}} \Z_{n}
	 +  \cc \cdot  n^{(3\kappa/2)-1} m^{1/2}.
	\end{eqnarray*}
Similarly, we get that 
\begin{eqnarray*}
		\sqrt{\frac{T}{m}} \DDD_{n} \geq 
	\sqrt{\frac{T}{m}} \Z_{n}
	 -  \cc \cdot  n^{(3\kappa/2)-1} m^{1/2}.
	\end{eqnarray*}
Therefore 
\begin{eqnarray*}
 	\P \left\{
		\sqrt{\frac{T}{m}}
		\DDD_{n}
		\leq u_{m}
	\right\}
	&\leq& 
 	\P \left\{
		\sqrt{\frac{T}{m}}
		\Z_{n}
		\leq u_{m} + \cc \cdot  n^{(3\kappa/2)-1} m^{1/2}
	\right\}
	+\frac{C}{m^{\beta}},\\
	 	\P \left\{
		\sqrt{\frac{T}{m}}
		\DDD_{n}
		\leq u_{m}
	\right\}
	&\geq& 
 	\P \left\{
		\sqrt{\frac{T}{m}}
		\Z_{n}
		\leq u_{m} - \cc \cdot  n^{(3\kappa/2)-1} m^{1/2}
	\right\}
	-\frac{C}{m^{\beta}}.
\end{eqnarray*}
This observation completes the proof of Theorem~\ref{thm35}.

\subsection{Taylor expansions for accompanying laws}
\label{Taylor}
In this subsection, we explain why it is impossible to get power rate of convergence  (``fast convergence'') of the maximal deviation distribution to the Gumbel distribution  looking at Theorem~\ref{thm35}. This result is very typical for such problems, see e.g. \cite{KP}.

Note that  \(A_{m}^{\pm}(y)\) for large \(m\) can be represented as 
\begin{eqnarray*}
	A_{m}^{\pm}(y) = g_{1}\left(
		y \pm \bar{c} n^{3 \kappa/2-1} m^{1/2}b_{m}
	\right) \cdot 
	g_{2}\left(
		y \pm \bar{c} n^{3 \kappa/2-1} m^{1/2}b_{m}
	\right),
\end{eqnarray*}
where
\begin{eqnarray*}
	g_{1}(y) := \exp \left\{ 
		-2e^{-y-y^{2} \lambda_{n}}
	\right\} ,\qquad 
	g_{2}(y) := e^{
	- 2 m\left(
		1-\Phi \left(
			u_{m}(y) \sqrt{ (b-a)/J}
		\right)
	\right)},
\end{eqnarray*}
and \(\lambda_{n}:=1/\left( 4 \ln(h_{1} m) \right)\).
Let us apply several times the Taylor theorem to the function \(g_{1}(\cdot)\). First, taking into account that \(\bar{c} n^{3 \kappa/2-1} m^{1/2}b_{m} \to 0 \) as \(m \to \infty\), we expand \(g_{1}(\cdot)\) in the  vicinity of \(y\):
\begin{multline}
\label{multi1}
g_{1}^{\pm}(y):=
	 g_{1}\left(
		y \pm \bar{c} n^{3 \kappa/2-1} m^{1/2}b_{m}
	\right) 
	\\=
	g_{1}(y) + 2 g_{1}(y) e^{-y-y^{2} \lambda_{n}} \left(
		1 - 2 y \lambda_{n}
	\right) \bar{c} n^{3 \kappa/2-1} m^{1/2}b_{m} \left( 1 + o(1) \right)
	\\=  g_{1}(y) + c_{1} m^{-r}(1+o(1))
\end{multline}
with some positive \(c_{1}\) and \(r\). Next, applying the Taylor theorem to \(g_{1}(\cdot)\) near zero, we get 
\begin{multline}
\label{multi2}
g_{1}(y) = \exp\left\{
	-2e^{-y} \left(
		1 - y^{2} \lambda_{n}(1+ o(1))
	\right)
\right\}
\\
=
e^{-2e^{-y}} \left(
	1 + 2\lambda_{n} e^{-y} y^{2} (1 + o(1) ) 
\right) = e^{-2e^{-y}}  + \frac{c_{2}}{\ln m} (1+ o(1)).
\end{multline}
Combining \eqref{multi1} with \eqref{multi2}, we get that  the difference between \(g_{1}^{\pm}(y)\) and \(e^{-2e^{-y}}\) is of the order \(c_{2} (\ln m)^{-1} (1+ o(1))\). As for the function \(g_{2}(\cdot)\), we note that 
\begin{multline*}
g_{2}^{\pm}(y):=
	 g_{2}\left(
		y \pm \bar{c} n^{3 \kappa/2-1} m^{1/2}b_{m}
	\right)
	\\=
	g_{2} (y) \pm c m g_{2} (y) \cdot  p
	\left( 
	u_{m}^{\pm}(y) \frac{ b-a}{J} 	
	\right)
	\frac{  n^{3 \kappa/2-1} m^{1/2}} {2 h_{2}} (1 + o(1)),
	\end{multline*}
where by \(p(\cdot)\) we denote the density of the standard normal distribution, and 
\[u_{m}^{\pm}(y)  = 
	\frac{
			y \pm \bar{c} n^{3 \kappa/2-1} m^{1/2} b_{m}
		} {
			2 h_{2} b_{m}
		} 
		+
		 b_{m}.
\]
It is a worth mentioning that 
\[
	 p
	\left( 
		u_{m}^{\pm}(y) \frac{ b-a}{J} 	
	\right) = 
	\frac{1}{\sqrt{2 \pi}} e^{ - (u_{m}^{\pm}(y))^{2}(b-a) / (2J) }
	=\frac{c}{m} (1 + o(1)),
\]
because in the trigonometric case \(b-a=2 J h_{2}\). Therefore, 
\[
 g_{2}\left(
		y \pm \bar{c} n^{3 \kappa/2-1} m^{1/2}b_{m}
	\right)
	=
	g_{2} (y) + \frac{c}{m^{s}} (1 + o(1))
\]
with some positive \(c\) and \(s\). Finally, applying \eqref{aphi},  we arrive at 
\[
g_{2}(y) = \exp\left\{-m \frac{1}{\sqrt{2 \pi} u_{m}}e^{-u_{m}^{2} (b-a)/(2J)} \right\}
=
\left(
	1+ \frac{1}{\sqrt{\ln m}}
\right) (1 + o(1)),
\]
because \(b-a=2 J h_{2}\). Therefore, 
\[
A_{m}^{\pm}(y)  = e^{-2 e^{-y}}+\frac{c}{\sqrt{\ln m}} (1+o(1)), \qquad m \to \infty,
\]
 and the desired result \eqref{unverbesser} follows.
\section{Legendre polynomials}
\label{Leg}
In this section, we consider the orthogonal Legendre polynomials $P_{n}(x),-1\leq x\leq 1,$
defined by the formula%
\begin{equation}
P_{n}(x)=\frac{1}{n! 2^{n}}\left[ \left( x^{2}-1\right) ^{n}\right] ^{(n)}, \qquad n=0, 1, 2, ... .
\label{leg1}
\end{equation}
The orthonormal Legendre polynomials \ $\widehat{P}_{n}(x)$ are defined by 
\begin{equation}
\widehat{P}_{n}(x)=\sqrt{\frac{2n+1}{2}}P_{n}(x), \qquad -1\leq x\leq 1.  \label{2}
\end{equation}
Recall that we are interested in the asympotic behaviour of the probability of the event
\(
 \left\{ 
	\bz\geq u \; c\; \sqrt{\delta / 2}
 \right\}
\)
as \(u \to +\infty\), where
\begin{equation}
\bz = \bz^{J, m} 
\triangleq \sup_{x\in \lbrack -1,1]}\bp^{J, m} (x),  
\label{6}
\end{equation}%
and the Gaussian field $\bp^{J,m}(x)$ is defined by 
\begin{equation}
\bp^{J, m}(x)=\bp(x) \triangleq c \; \sum_{j=0}^{J}\sqrt{\frac{2j+1}{2}}P_{j}(x)Z_{j} = 
c\; \sum_{j=0}^{J}\widehat{P}_{j}(x)Z_{j},  \quad x \in [-1,1],
\label{7}
\end{equation}%
$Z_{j}$ are i.i.d. standard normal random variables, and \(c\) is a constant that will be defined later. Note that here we slightly change the notation introduced in Section~\ref{mainres}.

This section is devoted to the proof of  Theorem~\ref{main} (ii). The main idea of the proof is to use Corollary~8.3 from  \cite{Piterbarg}. Below we check the conditions of this theorem, listed on pp.118-119 and 133 from \cite{Piterbarg}. 

\subsection{Properties of  Legendre polynomials}
Below we list some important properties of the Legendre polynomials, as well as some essential ideas of the proofs of these properties.
\begin{enumerate}
\item
Many properties of  Legendre polynomials may be derived from the generating
function representations 
\[
F(x,w)\triangleq \frac{1}{\sqrt{1-2xw+w^{2}}}=\sum_{n=0}^{\infty
}P_{n}(x)w^{n}, 
\]%
\[
F_{x}^{\prime }(x,w)=w\left( 1-2xw+w^{2}\right)
^{-3/2}=\sum_{n=0}^{\infty }P_{n}^{\prime }(x)w^{n},
\]%
see \cite{Suetin}, formula (11). In particular, 
\begin{eqnarray}
P_{n}(1)=1,  &\quad& P_{n}(-1)=(-1)^{n}, \nonumber\\
P_{n}^{\prime }(1)=\frac{n(n+1)}{2}, &\quad& P_{n}^{\prime }(-1)=\left( -1\right) ^{n+1}\frac{n(n+1)}{2}.  \label{3}
\end{eqnarray}%
\item
From the Laplace formula (\cite{Suetin}, p. 128, formula (1))%
\[
P_{n}(x)=\frac{1}{\pi }\int_{0}^{\pi }\left( x+i\sqrt{1-x^{2}}\cos \theta
\right) ^{n}d\theta ,\text{ }\left\vert x\right\vert \leq 1, 
\]%
we derive that%
\begin{equation}
\left\vert P_{n}(x)\right\vert \leq \frac{1}{\pi }\int_{0}^{\pi
}[x^{2}+(1-x^{2})\cos ^{2}\theta ]^{n/2}d\theta \leq 1,\text{ }\left\vert
x\right\vert \leq 1,  \label{4}
\end{equation}%
and 
\begin{equation}
\left\vert P_{n}(x)\right\vert <1\text{ for }\left\vert x\right\vert <1.
\label{pn5}
\end{equation}
\item Another important property is that \(\max_{\lbrack -1,1]}\left\vert P_{j}^{\prime }(x)\right\vert =
P_{j}^{\prime }(1)\). In fact, by 8.915 (2) from \cite{gr},
\begin{equation*}
P_{j}^{\prime }(x)=\sum_{\{k:\text{ }j-2k-1\geq 0\}}(2j-4k-1)P_{j-2k-1}(x).
\end{equation*}%
Clearly, 
\begin{equation}
\left\vert P_{j}^{\prime }(x)\right\vert \leq \sum_{\{k:\text{ }j-2k-1\geq
0\}}(2j-4k-1).  \label{2a}
\end{equation}
Suppose that $j=2l.$ Then we obtain from \eqref{2a}
\begin{eqnarray}
\nonumber
\left\vert P_{2l}^{\prime }(x)\right\vert &\leq&
(4l-1)+(4l-5)+...+(4l-4(l-1)-1)=4l^{2}-\sum_{k=0}^{l-1}(4k+1)\\
\label{3a}
&=&
4l^{2}-2l(l-1)-l=2l^{2}+l=\frac{2l(2l+1)}{2}=P_{2l}^{\prime }(1).  
\end{eqnarray}
If $j=2l-1$ we obtain again from (\ref{2a})%
\begin{eqnarray}
\nonumber
\left\vert P_{2l-1}^{\prime }(x)\right\vert &\leq&
(4l-3)+(4l-7)+...+(4l-4(l-1)-3)\\
\nonumber
&=&
4l^{2}-\sum_{k=0}^{l-1}(4k+3)=4l^{2}-2l(l-1)-3l=\\
\label{4a}
&=& 2l^{2}-l=\frac{(2l-1)2l}{2}%
=P_{2l-1}^{\prime }(1).  
\end{eqnarray}
Therefore, from  (\ref{3a}) and (\ref{4a}) it follows that
\[
\max_{\lbrack -1,1]}\left\vert P_{j}^{\prime }(x)\right\vert =
P_{j}^{\prime }(1)=\frac{j(j+1)}{2}.
\]

\end{enumerate}
\subsection{Covariance function of the  process $\p(x)$}
 The covariance function $%
r(x,y)$ of the process \(\p(x)\) is equal to%
\[
r(x,y)=c^{2} \; \sum_{j=0}^{J} \frac{2j+1}{2} P_{j}(x)P_{j}(y),
\]%
while the correlation function $\rho (x,y)$ is equal to 
\begin{multline}
\rho (x,y)=\frac{\sum_{j=0}^{J}\left( 2j+1\right) P_{j}(x)P_{j}(y)}{\left(
\sum_{j=0}^{J}\left( 2j+1\right) P_{j}^{2}(x)
\right) ^{1/2} \left(\sum_{j=0}^{J}\left(
2j+1\right) P_{j}^{2}(y)\right) ^{1/2}},\\ \rho (x,x)=1, \quad x\in \lbrack -1,1].
\label{8}
\end{multline}
Denote by $\sigma^{2}(x)$ the variance of the process $\p (x)$, that is,
\begin{eqnarray}
\label{sigmaleg}
\sigma ^{2}(x)=r(x,x)=c^{2} \; \sum_{j=0}^{J}\frac{2j+1}{2}P_{j}^{2}(x).
\end{eqnarray}
Let us check the condition (E1) on p. 118 from  \cite{Piterbarg}, which states that the variance of the process \(\p(x)\) can be represented as 
\begin{eqnarray}
	\label{e1} 
	\sigma(x) = 1 - \left|
		A_{j}(x -x_{j})
	\right|^{\beta_{j}}
	\left(
		1 + o(1)
	\right), \quad j=1,..,q,
\end{eqnarray}
where   \(x_{1}, ... ,x_{q}\) are the points of maximum of \(\sigma^{2}(x)\), \(A_{j} \ne 0\),   \(\beta_{j} \in \R\), \(j=1..q\). From (\ref{4}) and (\ref{pn5}), it follows that  \(q=2\), $x_{1}=1$ and $x_{2}=-1$, and 
\begin{eqnarray}
\label{scor}
	\sigma ^{2}(1)=\sigma
^{2}(-1)\triangleq \sigma _{\max }^{2}=c^{2} \; \sum_{j=0}^{J}\frac{2j+1}{2}=	\frac{c^{2}}{2}(J+1)^{2}.
\end{eqnarray}
Since from \eqref{e1} it follows that \(\sigma_{\max}=1\), we choose \(c=\sqrt{2} / (J+1)\). Next, we calculate the derivative of the variance, 
\begin{equation*}
\sigma ^{\prime }(x)=\frac{c}{\sqrt{2\sum_{j=0}^{J}\left( 2j+1\right)
P_{j}^{2}(x)}}\sum_{j=0}^{J}\left( 2j+1\right) P_{j}(x)P_{j}^{\prime }(x).
\end{equation*}%
In the neighbourhood of $x=1$ we have an expansion%
\begin{equation}
\sigma (x)=\sigma (1)-\sigma ^{\prime }(1)(1-x)+o((1-x)^{2}),  \label{10}
\end{equation}%
where 
\begin{equation*}
\sigma ^{\prime }(1)=\frac{c}{4\sqrt{2}}J(J+1)(J+2) =
J(J+2) / 4.  \label{11}
\end{equation*}%
Analogously, in the neighbourhood of $x=-1$ we have an expansion%
\begin{equation}
\sigma (x)=\sigma (-1)+\sigma ^{\prime }(-1)(x+1)+o((x+1)^{2}),  \label{12}
\end{equation}%
where%
\begin{equation*}
\sigma ^{\prime }(-1)=-J(J+2) / 4.
\end{equation*}%
From (\ref{10}) and (\ref{12}), we finally conclude that \eqref{e1} holds with 
$\beta_{1} =\beta_{2}=1$ and $A_{1}=A_{2}=J(J+2) / 4$.
\subsection{Local homogeneity}
The second condition (E2) on p.112  from \cite{Piterbarg} is about the local homogeneity of the process. It states that for any \(j=1..q\),
\begin{eqnarray}
\label{e2}
	\rho(x,y) = 1 - C_{j} |x-y|^{\alpha_{j}} 
	\left(
		1+ o(1) 
	\right), \qquad \mbox{as} \quad x \to x_{j}, \; y \to x_{j}, 
\end{eqnarray}
where \(C_{j} \ne 0\), \(\alpha_{j} \in \R\). Let us first check \eqref{e2} for the point \(x_{1}=1\). Denote 
\begin{equation*}
S_{lk}(x) = S_{lk}^{J}(x) :=\sum_{j=0}^{J}\left( 2j+1\right)
P_{j}^{(l)}(x)P_{j}^{(k)}(x), \qquad k=0,1,2, \;l=0,1. 
\end{equation*}
Then we get from (\ref{8})  with \(\Delta :=y-x\), 
\begin{eqnarray*}
\rho (x,y)&=&\rho (x,x+\Delta )\\
&=&\frac{\sum_{j=0}^{J}\left( 2j+1\right)
P_{j}(x)P_{j}(x+\Delta )}{\left[
\left( \sum_{j=0}^{J}\left( 2j+1\right)
P_{j}^{2}(x)
\right)
\left(
\sum_{j=0}^{J}\left( 2j+1\right) P_{j}^{2}(x+\Delta )\right) \right]
^{1/2}}\\
&=&
\frac{S_{00}(x)+S_{01}(x)\Delta +\frac{1}{2}S_{02}(x)\Delta
^{2}+o(\Delta ^{3})}{\left( \bar{S}(x)\right) ^{1/2}},
\end{eqnarray*}
where we use the notation
\begin{eqnarray*}
\bar{S}(x)=
\left[ S_{00}(x)\right]
^{2}+2S_{00}(x)S_{01}(x)\Delta +S_{00}(x)\left[
S_{11}(x)+S_{02}(x)\right] \Delta ^{2}+o(\Delta ^{3}).
\end{eqnarray*}
Continuing the line of reasoning, 
\begin{eqnarray*}
\rho (x,y)&=&
\frac{1+\frac{S_{01}(x)}{S_{00}(x)}\Delta +\frac{1}{2}\frac{%
S_{02}(x)}{S_{00}(x)}\Delta ^{2}+o(\Delta ^{3})}{\left( 1+2\frac{%
S_{01}(x)}{S_{00}(x)}\Delta +\frac{\left[ S_{11}(x)+S_{02}(x)%
\right] }{S_{00}(x)}\Delta ^{2}+o(\Delta ^{3})\right) ^{1/2}}
\\
&=&
\frac{1+\frac{S_{01}(x)}{S_{00}(x)}\Delta +\frac{1}{2}\frac{%
S_{02}(x)}{S_{00}(x)}\Delta ^{2}+o(\Delta ^{3})}{1+\frac{%
S_{01}(x)}{S_{00}(x)}\Delta +\frac{1}{2}\left\{ \frac{\left[
S_{11}(x)+S_{02}(x)\right] }{S_{00}(x)}-\left( \frac{%
S_{01}(x)}{S_{00}(x)}\right) ^{2}\right\} \Delta ^{2}+o(\Delta ^{3})}\\
&=&
\frac{1+\frac{S_{01}(x)}{S_{00}(x)}\Delta +\frac{1}{2}\frac{%
S_{02}(x)}{S_{00}(x)}\Delta ^{2}}{1+\frac{S_{01}(x)}{%
S_{00}(x)}\Delta +\frac{1}{2}\left\{ \frac{\left[
S_{11}(x)+S_{02}(x)\right] }{S_{00}(x)}-\left( \frac{%
S_{01}(x)}{S_{00}(x)}\right) ^{2}\right\} \Delta ^{2}}(1+o(\Delta
^{3}))\\
&=&
1-\frac{1}{2}\left( \frac{S_{11}(x)}{S_{00}(x)}-\left( \frac{%
S_{01}(x)}{S_{00}(x)}\right) ^{2}\right) (y-x)^{2}+o((y-x)^{2}).
\end{eqnarray*}
We have from (\ref{3})
\begin{eqnarray*}
S_{00}(1)&=& \sum_{j=0}^{J}\left( 2j+1\right) P_{j}^{2}(1)=\left(
J+1\right) ^{2},\\
S_{01}(1)&=&\sum_{j=0}^{J}\left( 2j+1\right) P_{j}(1)P_{j}^{\prime
}(1)\\
&=& \sum_{j=1}^{J}j^{3}+\frac{3}{2}\sum_{j=1}^{J}j^{2}+\frac{1}{2}%
\sum_{j=1}^{J}j=\frac{J(J+1)^{2}(J+2)}{4},\\
S_{11}(1) &=& \sum_{j=0}^{J}\left( 2j+1\right) \left( P_{j}^{\prime
}(1)\right) ^{2}=\frac{1}{4}\sum_{j=0}^{J}\left( 2j+1\right) j^{2}(j+1)^{2}\\ 
&=&
\frac{1}{2}\sum_{j=1}^{J}j^{5}+\frac{5}{4}\sum_{j=1}^{J}j^{4}+%
\sum_{j=1}^{J}j^{3}+\frac{1}{4}\sum_{j=1}^{J}j^{2}=\frac{%
J^{2}(J+1)^{2}(J+2)^{2}}{12},
\end{eqnarray*}
Therefore, we get 
\begin{eqnarray*}
\frac{1}{2}\frac{S_{11}(1)}{S_{00}(1)}&=&\frac{J^{2}(J+2)^{2}}{24},\\
\frac{1}{2}\left( \frac{S_{01}(t)}{S_{00}(t)}\right) ^{2}&=&\frac{%
J^{2}(J+2)^{2}}{32},\\
\frac{1}{2}\left( \frac{S_{11}(1)}{S_{00}(1)}-\left( \frac{%
S_{01}(1)}{S_{00}(1)}\right) ^{2}\right)& =&\frac{J^{2}(J+2)^{2}}{24}-%
\frac{J^{2}(J+2)^{2}}{32}=\frac{J^{2}(J+2)^{2}}{96}.  
\end{eqnarray*}
Thus the condition \eqref{e2}  holds near the point $x_{1}=1$ with $\alpha_{1}=2$ and $C_{1}=J^{2}(J+2)^{2} / 96.$ It follows
from (\ref{3}) that $S_{00}(-1)=S_{00}(1),$ $%
S_{01}(-1)=-S_{01}(1),$ $S_{11}(-1)=S_{11}(1)$ and, hence,
the same expansion holds near point $x=-1,$ with  $\alpha_{2} =\alpha_{1}$ and $C_{2}=C_{1}.$ 
\subsection{Global H{\"o}lder condition}
\label{ho}
We now check the condition (E3) on p. 118 from  \cite{Piterbarg},  which states there exist some \(g>0\),  \(G>0\), such that for all \(x,y\), 
\begin{eqnarray*}
	\E\left( 
		\bp^{J,m}(x) - \bp^{J,m}(y) 
	\right)^{2}
	\leq 
	G \left|
		x-y
	\right|^{g}.
\end{eqnarray*}
This condition immediately follows from 
\begin{eqnarray*}
	d^{2}(x,y)&=&c^{2} \cdot \E \left( \sum_{j=0}^{J}\sqrt{\frac{2j+1}{2}}\left(
	P_{j}(x)-P_{j}(y)\right) Z_{j}\right) ^{2}\\
		&=& c^2
\sum_{j=0}^{J}\frac{2j+1}{2}%
	\left( P_{j}(x)-P_{j}(y)\right) ^{2}\leq C_{J}(x-y)^{2}
\end{eqnarray*}
with some constant \(C_J\) depending only on \(J\), because  
\[
	\left\vert P_{j}^{\prime
	}(x)\right\vert \leq P_{j}^{\prime }(1)=j(j+1) / 2.
\]
\subsection{Concluding remarks}
Applying Corollary~8.3 from \cite{Piterbarg} (p.~134), we arrive at  
\begin{equation}
\P \left\{ 
	\sup_{x \in [-1,1]} \bp^{J, m}(x) >u
\right\} 
=
2 \left( 
	1 - \Phi\left( 
		u
	\right)
\right) 
\left(
	1 + o(1)
\right), \qquad u \to \infty.
\label{17}
\end{equation}
Intuitively (\ref{17}) means that for the large level $u$ the considered probability  is approximately equal to the sum of two probabilities\ $\P\{ \bp^{J, m}(1)\geq u\}$ and $\P\{ \bp^{J,m}(-1)\geq u\}$. Taking into account \eqref{aphi}, we get
\begin{equation*}
\P \left\{ 
	\sup_{x \in [-1,1]} \bp^{J, m}(x) >u
\right\} 
=
	\frac{\sqrt{2}}{\sqrt{\pi} u } e^{-u^2 /2 }
\left(
	1 + o(1)
\right), \qquad u \to \infty.
\end{equation*}
Finally, we conclude that 
\begin{eqnarray*}
\P \left\{ 
	\tzeta >u
\right\} 
&=&
\P \left\{ 
	\sup_{x \in [-1,1]} \bp^{J, m}(x) >u c \sqrt{\delta/2}
\right\}\\ 
&=&
	\frac{2}{(\sqrt{\pi} c) \cdot (\sqrt{\delta} u) } e^{-c^2  \delta u^2/4 }
\left(
	1 + o(1)
\right), \qquad u \to \infty.
\end{eqnarray*}
This observation completes the proof of Theorem~\ref{main} (ii).

\section{Wavelets} 
\label{wav}
In this section, we consider the case of Haar wavelets, that is, 
	\begin{eqnarray*}
		\psi_{j} (x)  =\Biggl\{  \frac{1}{\sqrt{\delta}}, \qquad \frac{1}{\sqrt{\delta}}\Biggl(I\{x \in [a+\delta/2, a+\delta ]\} - I\{x \in [a,a+\delta/2]\} \Biggr)
		\Biggr\},
	\end{eqnarray*}
where \(\delta = 2^{-l}\) for some \(l \in \N\). For this set of functions,
\begin{eqnarray*}
	\tzeta = 2^{l/2} \left(
		Z_{0} + | Z_{1} |
	\right).
\end{eqnarray*}
Therefore \(\tzeta/2^{l/2}\) has distribution with density and cdf equal to 
\[ 
	p_{Z_{0} + | Z_{1} |} (x) = \frac{1}{\sqrt{\pi}}
	e^{-x^{2}/4} \Phi \left(
		x/\sqrt{2}
	\right),
	\qquad 
	F_{Z_{0} + | Z_{1} |} (x) = 1 - \Phi^{2} \left(
		x/\sqrt{2}
	\right).
\]
Next, we apply a Taylor expansion of the function \( 1 - \Phi(x)\) for large \(x\) up to the second order \eqref{aphi}, and get that
\begin{eqnarray}
\label{urg}
	1 - \Phi^{2} (x/\sqrt{2}) = 
	\frac{2}{\sqrt{ \pi} x} e^{-x^{2}/2}
	\left(
		1 - \frac{2}{x^2} + 
			o\left( 
				\frac{1}{x^2}
			\right)
		\right).
\end{eqnarray}
This completes the proof of Theorem~\ref{main} (iii). 

\section{Discussion}
\label{disc}
The main contributions of this paper are Theorems~\ref{main3} and \ref{thm35}  , which give the asymptotic behavior of the maximal deviation distribution for projection estimates. Our research is  motivated by the paper~\cite{Fig3}, which is to the best of our knowledge the unique publication on this topic. Below we list the main findings of our research:
\begin{enumerate}
	\item first of all, we provide a unified treatment for different sets of basis functions (Legendre polynomials, trigonometric basis, wavelets);
	\item we derive the asymptotics of the maximal deviation distribution for any Legendre polynomials, not only for piecewise constant and piecewise linear basis functions, as it was known before;
	\item we have found a sequence of accompanying laws, which leads to better convergence rate (see Theorem~\ref{thm35}).
\end{enumerate} 

The main ingredient of the proof of Theorem~\ref{main3} is formulated in Theorem~\ref{main}, which reveals the essential difference in asymptotic behaviour of Gaussian processes for different sets of basis functions. An open problem is to prove similar facts for any basis under some mild conditions. The existing theory for Gaussian processes doesn't have a unified remedy for  solving such problems, and  therefore this issue can be a good topic for further research.

\appendix
\section{Some proofs}
\subsection{Proof of  Proposition~\ref{FG}}
\label{prooffg}
The proof follows the same lines as the proof of Theorem~4.1 in \cite{Fig3}.

\textbf{1.} \textit{Preliminary remarks.} For an empirical process \(G(\cdot)\) and positive constant \(\kappa\), introduce the notation
\[
	\L(\kappa, G, d; x) := \kappa \sum_{r=1}^{d} 
	\left[
		\int_{D}	\varphi_{r} (u) \; dG(u)
	\right] 
	\varphi_{r} (x).
\]
When there is no risk of confusion, we will use the simplified notation \(\L(\kappa, G)\).
Note that 
\[
	\hs_n(x)  - \E\hs_n (x) = \L\Bigl(\sqrt{n} / T, Z _{n}(F_{\Delta}(\cdot)), d; x\Bigr)=:\L_{1}(x),
\]
where \(\Z_{n}(\cdot)\)  is  the standardized empirical process of a uniform random sample 
\(F_{\Delta}( 
				X_{\Delta}^{(k)}
			), \; k=1..n\), that is 
\[	
	\Z_{n}(x) := \sqrt{n} \Biggl( 
		 \frac{1}{n} \sum_{k=1}^{n }I \left\{ 
		 	F_{\Delta}( 
				X_{\Delta}^{(k)}
			) \leq x
		\right\} - 
		\P\left\{F_{\Delta}( 
				X_{\Delta}^{(k)}
			)\leq x \right\}
	\Biggr),
\]
and \(F_{t}(\cdot)\) is the distribution function of \(X_{t}\). 
In our notations,  \(Z _{n}(F_{\Delta}(\cdot))\) is the empirical process for \(X_{\Delta}^{(k)}, \; k=1..n\), 
\begin{eqnarray*}
	Z _{n}(F_{\Delta}(x)) &= &
	\sqrt{n} \Biggl(
		\frac{1}{n} \sum_{k=1}^{n } I \left\{ X_{\Delta}^{(k)}\leq x\right\} - 	F_{\Delta}(x) 
	\Biggr).
\end{eqnarray*}

\textbf{2.} \textit{Koml{\'o}s-Major-Tusnady construction.} \label{KMT}
 Applying Theorem~3 in \cite{KMT}, we get that there exists a version of \(\Z_{n}(x)\) (denoting below also as \(Z_{n}(x)\) for simplifyng the notation) such that  for any \(y >0\)  the probability of the event 
\begin{eqnarray*}
	\W_{n}(y) := \left\{
		\sup_{x \in [0,1]} \left| 
			Z_{n}(x) - B_{n}(x) 
		\right| 
		\leq
		C _{1}\frac{\log(n)}{\sqrt{n}} + y
	\right\}
\end{eqnarray*}
is larger than  \(1- K e^{-\lambda y \sqrt{n}}\), 
where  \(B_{n}(x) = W_{n}(x) -  x W_{n}(1)\) is the corresponding Brownian bridge constructed by Brownian motion \(W_{n}(x)\) and \(C_{1},K,\lambda\) are some positive absolute constants. As usual, we take \(y=y^{*}=\log(n) / \sqrt{n}\), and get that the set \(\W^{*}_{n}=\W_{n}(y^{*})\) defined by 
\begin{eqnarray}
\label{Wstar}
\W^{*}_{n} := \left\{
		\sup_{x \in [0,1]} \left| 
			Z_{n}(x) - B_{n}(x) 
		\right| 
		\leq
		C _{2}\frac{\log(n)}{\sqrt{n}}
	\right\}
\end{eqnarray}
is of probability larger than \(1-K/n^{\lambda}\).

\textbf{3.} \(\textcolor{blue}{
\L_{1}(x) \too \L_{2}(x):=\L\Bigl(\sqrt{n} / T, B _{n}(F_{\Delta}(\cdot)), d; x\Bigr).}\)

By the definition of the functional \(\L\),  
\begin{eqnarray}
\label{L1}
	\L_{1}(x) -\L_{2}(x)  =
	\L\Bigl(\sqrt{n} / T, Z _{n}(F_{\Delta}(\cdot))- B _{n}(F_{\Delta}(\cdot))\Bigr)
\end{eqnarray}
It is a worth mentioning for any empirical process \(G\), 
\begin{eqnarray}
\label{w}
	\sup_{x \in D} \left| \L\Bigl(\kappa, G, d; x \Bigr) \right|
	\leq 
	C_{2}\; \kappa \;m \;w(G, D, \delta),
\end{eqnarray}
where \(C_{2}\) is a positive constant depending on \(\left( \varphi_{r} \right)\), and \(w\) is the modulus of continuity, that is 
\begin{eqnarray*}
 	w (G, D, \delta) := \sup\Bigl\{ 
		\left| 
			G(u) - G(v)
		\right| : \; 
		u,v \in D, \; |u-v| < \delta
	\Bigr\}.
\end{eqnarray*}
In fact, using  integration by parts we  get that for \(x \in I_{p}\), 
\begin{eqnarray*}
	\left|
		\L\Bigl(\kappa, G, d; x \Bigr)  
	\right|
	&=& 
	\left|
	\kappa \sum_{j=0}^{J} 
	\left[
		\int_{a}^{a+\delta}	\psi_{j} (u) \; dG\left(u+\bd\right)
	\right] 
	\psi_{j} (x-\bd)
	\right|
		\\
	&=&
	\left| 
	\kappa \sum_{j=0}^{J} \Bigl[
		\psi_{j} (a+\delta) \Bigl(
			G(a+\bd +\delta) - G (a+\bd ) 
		\Bigr) \Bigr. \right.
		\\&&
		\Bigl.\left.
		\hspace{1cm}
		-
		\int_{a}^{a+\delta} 
			\Bigl(
				G(u+\bd)-G(a+\bd)
			\Bigr) 
		d\psi_{j}(u)
		\Bigr]
		\psi_{j}(x-\bd)
		\right|\\
		&\leq&
		\kappa \sum_{j=0}^{J} 
		\left( 
			\sup_{x \in I_{1}}|\psi_{j} (x) |
			+ 
			V_{a}^{a+\delta} (\psi_{j})
		\right)
		\sup_{x \in I_{1}}|\psi_{j} (x) | \cdot
		w(G, D, \delta)\\
		&\leq& 
		C_{2}\; \kappa \;m \;w(G, D, \delta)
		,
\end{eqnarray*}
where \(\bar{\delta}: =\delta(p-1)\), \(C_{2}>0\), and we use a slightly simplified notation \(\psi_{j}(\cdot) := \psi_{j}^{m} (\cdot)\) and the conditions \eqref{cond1}. 
Combining \eqref{Wstar}, \eqref{L1} and \eqref{w} we get that on the set \(\W^{*}_{n}\), 
\begin{eqnarray*}
	\sup_{x \in D } \left| 
		\L_{1}(x) - \L_{2}(x) 
	\right| 
	 \leq 
	C_{3} \frac{m \log(n)}{T}, \qquad \mbox{where} 
	\;  C_{3}>0. 
\end{eqnarray*}
\textbf{4.} \(\textcolor{blue}{\L_{2}(x) \too \L_{3}(x):= 	\L\Bigl(\sqrt{n} / T, B _{n}(1-F_{\Delta}(\cdot)), d  ; x\Bigr)}\).

Taking into account that \(B_{n}(x) \eqd B_{n}(1-x)\) for any \(x \in [0,1]\), we get 
\begin{eqnarray*}
	\L_{2}(x)= \L\Bigl(\sqrt{n} / T, B _{n}(F_{\Delta}(\cdot))\Bigr) \eqd \L\Bigl(\sqrt{n} / T, B _{n}(1-F_{\Delta}(\cdot))\Bigr) = \L_{3}(x).
\end{eqnarray*}
\textbf{5.} \(\textcolor{blue}{\L_{3}(x) \too \L_{4}(x):= \L\Bigl(\sqrt{n} / T, W _{n}(1-F_{\Delta}(\cdot))\Bigr).}
\)

Obviously,
\begin{eqnarray*}
	\L_{4}(x) -\L_{3}(x)  =
	\L\Bigl(
		\sqrt{n} / T, \left(1 - F_{\Delta}(\cdot)\right) W_{n}(1)
	\Bigr)
\end{eqnarray*}
Similarly to Step 3, we get 
\begin{eqnarray*}
	\sup_{x \in D } \left| 
		\L_{4}(x) - \L_{3}(x) 
	\right|  &\leq& 
	\frac{\sqrt{n}d }{T}\; C_{2}\;
	w\Bigl(
		1-F_{\Delta}(\cdot), D, \delta
	\Bigr) 
	\;
	\left|W_{n}(1)\right|\\
	&\leq&\frac{\sqrt{n} d }{T}\;
	\Bigl(
		C_{5} \Delta^{2} + C_{6} \delta \Delta
	\Bigr) 
	\;
	\left|W_{n}(1)\right|\\
		&\leq&\frac{d }{\sqrt{n} }\;
	\Bigl(
		C_{5} \Delta + C_{6} \delta 
	\Bigr) 
	\;
	\left|W_{n}(1)\right|,
\end{eqnarray*}
where the second inequality holds due to an important sequence from \eqref{supp}: for any \(u<v, \; u,v \in D\),
\begin{eqnarray}
\nonumber
\left| 
 \P \left\{ 
	X_{\Delta} \geq u
\right\} - 
 \P \left\{ 
	X_{\Delta} \geq v
\right\}
\right| 
&<& q \Delta^{2} + \nu\left(
	[u,v]
\right) \Delta\\
\label{PP}
&\leq&
q \Delta^{2} 
+ 
\max_{x \in [u,v]}|s(x)| \cdot (v-u) \Delta.
\end{eqnarray}

\textbf{6.} \(\textcolor{blue}{\L_{4}(x) \too \L_{5}(x):=	\L\Bigl(1/ \sqrt{T}, W_{n}\left(\left(1-F_{\Delta}(x)\right)/\Delta\right),d; x \Bigr)}\).
  
 Applying properties of the Brownian motion, we get 
\begin{eqnarray*}
	\L\Bigl(\sqrt{n} / T, W _{n}(1-F_{\Delta}(\cdot))\Bigr)
	\eqd
	\L\Bigl(1/ \sqrt{T}, W_{n}\left(\left(1-F_{\Delta}(\cdot)\right)/\Delta\right)\Bigr).
\end{eqnarray*}

\textbf{7.} \(\textcolor{blue}{\L_{5}(x) \too \L_{6}(x):=\L\Bigl(1/ \sqrt{T}, W_{n}\left(\int_{\cdot}^{+\infty} s(u) du\right), d; x\Bigr)}\).  

Using the assumption \eqref{supp}, we get 
\begin{eqnarray*}
	\sup_{x \in D } \left| 
		\L_{5}(x) - \L_{6}(x) 
	\right| 
	\leq 
	\frac{dm}{\sqrt{T}} C_{7} w (W_{n}, D, q \Delta) %\\ \leq
	\frac{d }{\sqrt{T}}\; C_{2}\;
	w\left(
		W_{n}\Bigl(\left(1-F_{\Delta}(\cdot)\right)/\Delta\Bigr) - 
		W_{n}\Bigl(\int_{\cdot}^{+\infty} s(u) du\Bigr)
		, D, \delta
	\right). 
\end{eqnarray*}
In the paper \cite{FN}, it is proven that 
\begin{eqnarray*}
	\E \left[ 
		w (W_{n}, D, q \Delta)
	\right] \leq \breve{C} \sqrt{ \Delta \ln\left(\frac{1}{\Delta}\right)},
\end{eqnarray*}
and therefore due to Chebyshev inequality, 
\begin{eqnarray*}
\P \left\{ 
	\sup_{x \in D } \left| 
		\L_{5}(x) - \L_{6}(x) 
	\right| 
> \eps 
\right\} 
&\leq&
\P \left\{ 
	w (W_{n}, D, q \Delta)
	> \eps  \frac{\sqrt{T}}{dm C_{7}} 
\right\} \\
&\leq& 
C_{7} \breve{C}
 \frac{dm }{\eps  \sqrt{T}} 
 \sqrt{ \Delta \ln\left(\frac{1}{\Delta}\right)}.
\end{eqnarray*}

\textbf{8.} \(\textcolor{blue}{\L_{6}(x) \too \L_{7}(x):=\L\left(1/ \sqrt{T}, \int_{\cdot}^{+\infty} \sqrt{s(u)} dW_{n}(u), d; x\right)}\).  

The functionals \(\L_{6}(x)\) and \(\L_{7}(x)\) have the same distributions, because
\begin{eqnarray*}
	W_{n}\left(\int_{\cdot}^{+\infty} s(u) du \right)
	\eqd 
	\int_{\cdot}^{+\infty} \sqrt{s(u)} dW_{n}(u).
\end{eqnarray*}
Recall that 
\begin{eqnarray*}
\L\Bigl(1/ \sqrt{T}, \int_{x}^{+\infty} \sqrt{s(u)} dW_{n}(u) \Bigr) 
= 
\frac{1}{\sqrt{T}}
	\sum_{r=1}^{d} 
	\left[
		\int_{D}	\varphi_{r} (u) \sqrt{s(u)}\; dW_{n}(u)
	\right] 
	\varphi_{r} (x)
\end{eqnarray*}

\textbf{9.} \(\textcolor{blue}{\L_{7}(x) \too \L_{8}(x):=\L\left(\sqrt{(b-a)/T}, W_{n}(\cdot), d; x\right)}\).  
Let us show that 
\begin{eqnarray}
\label{L78} 
	\sup_{x \in D } \left| 
		\L_{8}(x)  - \sqrt{\frac{b-a}{s(x)}} \L_{7}(x) 
	\right| \leq
	C_{8}  T^{-1/2} \sup_{x \in D}
	 \left|
		W_{n}(x)
	\right|.
\end{eqnarray}
In fact, we can represent the difference in \eqref{L78} as 
\begin{eqnarray*}
	\L_{8}(x)  - \sqrt{\frac{b-a}{s(x)}} \L_{7}(x)  = 
	\sqrt{\frac{b-a}{T}}  
	\sum_{r=1}^{d} 
G_{r}(x) \varphi_{r} (x),
\end{eqnarray*}
where 
\begin{eqnarray*}
	G_{r}(x):=		\int_{D}	\varphi_{r} (u)  dW_{n}(u)
		-
		\int_{D}	\varphi_{r} (u) \sqrt{\frac{s(u)}{s(x)}} dW_{n}(u).
\end{eqnarray*}
Using integration by parts, we get 
\begin{multline*}
	G_{r}(x) =
	\varphi_{r}(b) \left(1- \sqrt{\frac{s(b)}{s(x)}}\right)  W_{n}(b) 
	-
	\varphi_{r}(a) \left(1- \sqrt{\frac{s(a)}{s(x)}}\right)  W_{n}(a)  \\
	+ 
	\int_{\R} \left(
		\varphi'_{r}(u) \left( 
			1-  \sqrt{\frac{s(u)}{s(x)}}
		\right)
		- \frac {
			\varphi_{r}(u) s'(u) 
		}
		{
			2 \sqrt{s(u) s(x)}
		}
	\right)
	W_{n} (u) du
\end{multline*}
Taking into account that \(s(u)\) as well as \(s'(u)\) are bounded on \(D\), we get 
\begin{eqnarray*}
	\sup_{D}|G_{r}(x)| \leq C_{9} \sup_{D} |W_{n} (u)|,
\end{eqnarray*}
and the required result follows. Note that the distribution of the random variable \(\sup|W_{n}(x)|\) is given by 
\begin{eqnarray*}
	\P \left\{
		\sup_{x \in D}|W_{n}(x)| 
		>
		u
	\right\} 
	&\leq& 
	\P \left\{
		\sup_{x \in D}W_{n}(x)
		>
		u
	\right\} 
	+
	\P \left\{
		\inf_{x \in D}W_{n}(x) 
		<
		-u
	\right\} \\
	&\leq& 
	2 \P \left\{
		\sup_{x \in D}W_{n}(x)
		>
		u
	\right\} \leq 
		2 \P \left\{
		\sup_{x \in [0,b]}W_{n}(x)
		>
		u
	\right\} \\
	 &\leq& 
	 4 \P \left\{
	 	W_{n}(b) >u
	\right\},
\end{eqnarray*}
see Theorem~2.18 (p.50) from \cite{MP}.

\textbf{10.}  \(\textcolor{blue}{\sup_{x \in D}\L_{8}(x) \too \mbox{ maximum of random variables}}\).  

Let us represent 
\begin{multline*}
	\L_{8}(x) = 
	\sqrt{\frac{b-a}{T}}
	\sum_{p=1}^{m} \sum_{j=0}^{J} 
	\Bigl[
		\int_{I_{p}}
		\psi_{j} (u-\delta(p-1)) 
		dW_{n}(u)
	\Bigr]
	\psi_{j}(x-\delta(p-1)), \; x \in I_{p}.
\end{multline*}
Note that \(Z_{j,p}:=\int_{I_{p}}
		\psi_{j} (u-\delta(p-1)) 
		dW_{n}(u)\) are the normal random variable with zero means and variances equal to 
	\[\int_{I_{p}}
		\psi^{2}_{j} (u-\delta(p-1)) 
		du
		=
		\int_{a}^{ a+\delta }
		\psi_{j}^{2}(u)
		du
	=1.\] 
Therefore,
\begin{eqnarray}
\label{L8}
	\sup_{x \in D} \left|
		\L_{8}(x) 
	\right| 
	= 	
	\sqrt{\frac{b-a}{T}}
	\max_{p=1..m}
	\left[
	\sup_{x \in I_{1}} \left|
		\sum_{j=0}^{J} Z_{j,p} \psi_{j}(x)
	\right|
	\right],
\end{eqnarray}
or, in other words, 
%\begin{eqnarray*}
\(
	\sup_{x \in D} \left|
		\L_{8}(x) 
	\right| =
	\sqrt{(b-a)/T} \cdot
	\max \left\{ \zeta_{1}, ..., \zeta_{m} \right\},
\)%\end{eqnarray*}
where \(\zeta_{1}, .. \zeta_{m}\) are independent copies of the random variable 
\begin{eqnarray*}
	\zeta = \zeta^{J,m} = \sup_{x \in I_{1}}\left|
		\sum_{j=0}^{J} Z_{j} \psi_{j}(x)
	\right|
\end{eqnarray*}
with  i.i.d. standard normal r.v.'s \(Z_{j},\; j=0..J\).

\textbf{11.} \textcolor{blue}{Last step.}

To complete the proof, we need the following technical lemma.
\begin{lem}
\label{final}
	Let \(\eta_{1}, ..., \eta_{k}\) be random variables such that
	\begin{eqnarray*}
		\P \Bigl\{ 
			\left| 
				\eta_{i+1} - \eta_{i}
			\right| 
			\leq \delta_{i}
		\Bigr\} \geq 1 -  \gamma_{i}, \quad i=1..(k-1),
	\end{eqnarray*}
for some positive \(\delta_{i}, \gamma_{i}, \; i=1..k\).  Denote by \(F_{\eta_{k}}\) the distribution function of \(\eta_{k}\).

Then 
\begin{eqnarray}
	F_{\eta_{k}}\left( 
		x- \sum_{j=1}^{k-1}\delta_{j}
	\right) 
	-
	\sum_{j=1}^{k-1}\gamma_{j}
	\leq 
	F_{\eta_{1}}(x) 
	\leq 
		F_{\eta_{k}}\left( 
		x	+ \sum_{j=1}^{k-1}\delta_{j}
	\right) 
	+
	\sum_{j=1}^{k-1}\gamma_{j}.
	\label{reslem}
\end{eqnarray}
\end{lem} 
\begin{proof}
	First note that 
\begin{eqnarray*}
	F_{\eta_{1}}(x) &=& \P\left\{
		\eta_{1}\leq x
	\right\} 
	\\
	&=&
	\P \left\{\eta_{1} \leq x, \eta_{2}\leq x+\delta_{1}\right\}
	+
	\P \left\{\eta_{1} \leq x, \eta_{2}>x+\delta_{1}\right\}\\
	&\leq&
		\P \left\{\eta_{2}\leq x+\delta_{1}\right\}
		+
		\P \left\{ 
			\left| 
				\eta_{2} - \eta_{1}
			\right|  > \delta_{1}
		\right\} \\
	&\leq&
	F_{\eta_{2}}(x+\delta_{1}) + \gamma_{1}.
\end{eqnarray*}
Analogously,
\begin{eqnarray*}
		F_{\eta_{2}}(x) 
		&\leq&
		F_{\eta_{3}}(x+\delta_{2}) + \gamma_{2},
\end{eqnarray*}
 and therefore
\begin{eqnarray*}
		F_{\eta_{1}}(x) 
		&\leq&
		F_{\eta_{3}}(x+\delta_{1}+\delta_{2}) + \gamma_{1}+\gamma_{2}.
\end{eqnarray*}
Continuing in this way, we obtain  the right hand side of \eqref{reslem}. As for the left side, we have
\begin{eqnarray*}
	F_{\eta_{2}}(x-\delta_{1}) &=& 
	\P \left\{\eta_{2} \leq x-\delta_{1}, \eta_{1}\leq x\right\}
		+
	\P \left\{\eta_{2} \leq x-\delta_{1}, \eta_{1} > x\right\}
	\\
	&\leq&
		\P \left\{\eta_{1}\leq x\right\}
		+
		\P \left\{ 
			\left| 
				\eta_{2} - \eta_{1}
			\right|  > \delta_{1}
		\right\} \\
	&\leq&
	F_{\eta_{1}}(x) + \gamma_{1}.
\end{eqnarray*}
 Hence, 
 \begin{eqnarray*}
	F_{\eta_{2}}(x-\delta_{1})  - \gamma_{1} \leq F_{\eta_{1}}(x).
\end{eqnarray*}
Next, 
 \begin{eqnarray*}
	F_{\eta_{3}}(x-\delta_{1}-\delta_{2})  &\leq& 
	F_{\eta_{2}}(x-\delta_{1})+\gamma_{2}\\
	&\leq& 
	F_{\eta_{1}}(x) + \gamma_{1}+ \gamma_{2},
\end{eqnarray*}
and therefore
\begin{eqnarray*}
	F_{\eta_{3}}(x-\delta_{1}-\delta_{2}) - \gamma_{1}-  \gamma_{2}
	\leq F_{\eta_{1}}(x).
\end{eqnarray*}
Continuing in this way, we get the left hand side of \eqref{reslem}.
\end{proof}

Returning to the proof of Proposition~\ref{FG}, we apply Lemma~\ref{final} with 
\begin{eqnarray*}
	\eta_{k} &:=&   \sup_{x \in D} \left\{ 
		\sqrt{\frac{b-a}{s(x)}} 
		\left| \L_{k}(x) \right|
	\right\}, \quad k=1..7,\\
	\eta_{8} &:=& 	\sqrt{\frac{b-a}{T}}
		\max_{p=1..m} \zeta_{p}.
\end{eqnarray*}  
Note that for all \(k=2,..,7\), 
\begin{eqnarray*}
	\left| 
		\eta_{k} - \eta_{k-1}
	\right| &\leq&   \sup_{x \in D} \left| 
		\sqrt{\frac{b-a}{s(x)}}
		\left|
			\L_{k}(x)
		\right| 
		-
		\sqrt{\frac{b-a}{s(x)}}
		\left|
			\L_{k-1}(x)
		\right| 
	\right|
	\\
	 &\leq&
	  \sup_{x \in D} \left\{ 
		\sqrt{\frac{b-a}{s(x)}}
	\cdot    \Bigl| 
		\L_{k}(x) - \L_{k-1}(x)
	 \Bigr|
	 \right\} 
	\\
	 &\leq&   \sup_{x \in D} \left\{ 
		\sqrt{\frac{b-a}{s(x)}}
	\right\} 
	\cdot   \sup_{x \in D} \Bigl| 
		\L_{k}(x) - \L_{k-1}(x)
	 \Bigr|,
\end{eqnarray*}  
and 
\begin{eqnarray*}
	\left| 
		\eta_{8} - \eta_{7}
	\right| &=&  
	\left| 
	\sup_{x \in D} \left\{ 
		\sqrt{\frac{b-a}{s(x)}} \left| \L_{7}(x) \right|
	\right\} 
	-
		\sqrt{\frac{b-a}{T}}
		\max_{p=1..m} \zeta_{p}
	\right|
	\\
	&=&  
	\left| 
	\sup_{x \in D} \left\{ 
		\sqrt{\frac{b-a}{s(x)}}  \left| \L_{7}(x) \right|
	\right\} 
	-
	\sup_{x \in D} 
	\left|
		\L_{8}(x)
	\right|
		\right|\\
	&\leq&
	\left|
		 \sup_{x \in D} \left\{ 
			\sqrt{\frac{b-a}{s(x)}}
		\L_{7}(x)
				\right\} 
		-
		\L_{8}(x)
	\right|,
\end{eqnarray*} 
where we apply \eqref{L8} in the second equality. Using the results obtaining on the previous steps of the proof (and changing for simplicity the indexes for constants),  we get that 
\begin{eqnarray*}
		\delta_{1}=C_{1} \frac{m \log  n}{T}, &&\qquad  \gamma_{1} = K/n^{\lambda},\\ 
		\delta_{3}=
		\Bigl(
			C_{2} \frac{ T m}{n^{3/2}} + \frac{C_{3}}{\sqrt{n}}
	\Bigr)  q_{n}^{(1)}, &&\qquad  \gamma_{3} = 2 (1 - \Phi(q_{n}^{(1)})),\\
		\delta_{5}=C_{4} 
		\frac{m}{\sqrt{T}} \sqrt{ \log \left(  \frac{ n}{T} \right) }, &&\qquad  \gamma_{5} = C_{5} \sqrt{T/n} ,\\
		\delta_{7}=C_{6} \frac{1}{\sqrt{T}} q_{n}^{(2)}, &&\qquad  \gamma_{7} = 4 (1 - \Phi(q_{n}^{(2)}/ \sqrt{b})),
\end{eqnarray*}
where the sequences  \(q_{n}^{(1)}, q_{n}^{(2)}\) are tending to \(\infty\) as \(n \to \infty\) and will be chosen in the sequel, and all other \(\delta\)'s and \(\gamma\)'s are equal to \(0\). Since for all positive \(x\), 
\begin{eqnarray}
\label{Michna}
1 - \Phi(x) \leq 	 \frac{ 1}{x\sqrt{2 \pi}} e^{-x^{2}/2},
\end{eqnarray}
see, e.g., p.2 in \cite{Michna}, we choose \(q_{n}^{(1)}=\sqrt{2 \lambda \ln n}\), and \(q_{n}^{(2)}=\sqrt{2 \lambda b \ln n}\), and  get that \(\sum_{i=1}^{k-1} \gamma_{i} \lesssim n^{-\lambda}\) as  \(n \to \infty\). Next, we set \(T=n^{\kappa}\). Note that the condition \(\kappa<1\)  guarantees that \(\Delta=T/n \to 0\) as \(n \to \infty\), and moreover 
\begin{eqnarray*}
	\delta_{5} &=& C  \frac{m}{n^{\kappa/2}} \sqrt{\log n},\\	
\delta_{1}&=&\frac{m \log n}{n^{\kappa}} \cdot O(1)  
	\lesssim 
	\delta_{5} \\
	\delta_{3}&=& 
	\frac{m \sqrt{\log n}}{n^{3/2-\kappa}} \cdot O(1)  
	+
	\frac{\sqrt{\log n}}{n^{1/2}} \cdot O(1)  
	\lesssim 
	\delta_{5},\\
	\delta_{7}&=&\frac{m \sqrt{\log n}}{n^{\kappa/2}} \cdot O(1)  
	\lesssim 
	\delta_{5}.
\end{eqnarray*}
Applying Lemma~\ref{final}, we arrive at the desired result.

\subsection{Proof of Corollary~\ref{corcor}} 
\label{acorcor}
\textbf{1.} For the cases of trigonometric basis and Legendre polynomials, we use the following result, which is given in the Russian edition of the book \cite{Piterbarg}  (1988, Corollary 6.4, page 74):

\begin{prop}
\label{propcorcor} Let $X(t),$ $t\in \lbrack 0,T]^{n},$ be
Gaussian field in $R^{n}$ with continuous trajectories. Suppose that 

\begin{enumerate}[(i)]
\item \qquad for
some $m,\sigma ^{2}<\infty$
\[
\left\vert \E X(t)\right\vert \leq m, \qquad VarX(t)\leq \sigma^{2};
\]
\item \qquad the global H{\"o}lder condition holds, that is for some \(C,\gamma >0\),
\[d_{X}^{2}(t,s)=\E(X(t)-X(s))^{2}\leq C\sum_{i=1}^{n}\left\vert
t_{i}-s_{i}\right\vert ^{\gamma }, \qquad \forall t,s;\]
\item \qquad 
 the correlation function $\rho (t,s)$ is separated from $-1
$, that is \[\rho (t,s)>-1+\delta\] for some positive $\delta .$ 
\end{enumerate}
Then there
exist $\rho >0$ such that%
\begin{multline*}
\sup_{t \in [0,T]}\P \left\{| X(t) |>u \right\}
=\sup_{t \in [0,T]}\P \left\{ X(t) > u \right\}
+
\sup_{t \in [0,T]}\P \left\{-X(t)>u \right\}
\\+O\left( \exp \left( -\frac{%
(u-m)^{2}}{2\sigma ^{2}}(1+\rho )\right) \right), \qquad u \to +\infty.   \label{1}
\end{multline*}
In particular, if additionally the process $X(t),t\in [0,T]$ has zero mean, then%
\begin{multline*}
\sup_{t \in [0,T]}\P \left\{| X(t) |>u \right\}
=2 \cdot \sup_{t \in [0,T]}\P \left\{ X(t) > u \right\}
+O\left( \exp \left( -\frac{%
u^{2}}{2\sigma ^{2}}(1+\rho )\right) \right) .  
\end{multline*}%
\end{prop}
\textbf{2.}
The process%
\begin{equation*}
X(t)=\widetilde{\p}(t)=\frac{1}{\sqrt{J}}\left\{ Z_{0}+\sqrt{2}\sum_{j=1}^{\frac{J-1}{2}}\left[
Z_{j}\cos (jt)+\widetilde{Z}_{j}\sin (jt)\right] \right\} 
\end{equation*}%
has zero mean and unit variance, therefore the condition (i) is fulfilled. Next,  we get that 
\begin{multline*}
	\E \left( X(t) - X(s) \right)^2 \\=
		\frac{2}{J} \cdot 
	\E \left[
		\sum_{j} Z_j \left(
			\cos(js) - \cos(jt)
		\right)
		+		
		\sum_{j} \tilde{Z}_j \left(
			\sin(js) - \sin(jt)
		\right)
	\right]^2,
\end{multline*}
and therefore the global H{\"o}lder  condition (ii) holds because 
\begin{eqnarray*}
\E \left( X(t) - X(s) \right)^2 
	&=&
	\frac{8}{J}
	\left(
		\sum_{j}
		\left(
			1 - \cos( j(t-s))
		\right)
	\right)^2\\
	&=&
	\frac{32}{J}
	\left(
		\sum_{j}
		\left(
			\sin^2 \left(
				\frac{j (t-s)}{2}
			\right)
		\right)
	\right)^2\\
&\leq&  \frac{2}{J} \left( \sum_j j^2 \right)^2 (t-s)^4.
\end{eqnarray*}
Finally, the covariance and the correlation functions are equal to 
\begin{equation}
\rho (t)=r(t)=\frac{1}{J}\left\{ 1+2\sum_{j=1}^{\frac{J-1}{2}}\cos
(jt)\right\} .
\end{equation}%
Clearly,%
\begin{equation}
\rho (t)\geq \frac{1}{J}\left\{ 1+2\sum_{j=1}^{\frac{J-1}{2}}(-1)\right\} =%
\frac{2}{J}-1>-1+\frac{1}{J}
\end{equation}%
so the condition (iii) is also satisfied. Therefore, we can apply Proposition~\ref{propcorcor} in the case of trigonometric basis.

\textbf{3.} Let us check the conditions (i) - (iii) for the case of Legendre polynomials, that is, for the Gaussian process
\begin{equation*}
\bp^{J, m}(x)=\bp(x) \triangleq c \; \sum_{j=0}^{J}\sqrt{\frac{2j+1}{2}}P_{j}(x)Z_{j} \quad x \in [-1,1],
\end{equation*}
where by \(P_j\) we denote the orthogonal polynomials defined by \eqref{leg1}. The condition (i) is trivial because the process \(\bp(x)\) has zero mean and bounded variance given by \eqref{sigmaleg}. The condition (ii) was already checked in Section~\ref{ho}. To show that the third condition holds, we consider the correlation function 
\[
\rho (x,y)=\frac{\sum_{j=0}^{J}\left( 2j+1\right) P_{j}(x)P_{j}(y)}{\left(
\sum_{j=0}^{J}\left( 2j+1\right) P_{j}^{2}(x)
\right) ^{1/2} \left(\sum_{j=0}^{J}\left(
2j+1\right) P_{j}^{2}(y)\right) ^{1/2}}.
\]
Applying the Cauchy-Schwarz inequality to the sequences \newline
\(\vec{a}=\left( P_{0}(x), \sqrt{3} P_{1}(x), \sqrt{5}P_{1}(x), ...\right)\) and 
\(\vec{b} =\left( P_{0}(y), \sqrt{3} P_{1}(y), \sqrt{5}P_{1}(y), ...\right)\), we immediately get that \(\rho (x,y) = -1\) if and only if \(\vec{a} = c \vec{b}\) for some negative \(c\). Since \(P_{0}(x)  = P_{0}(y)\), we conclude that \(\rho (x,y) > -1\) for any  \( (x,y)\) from the compact set \( [0,1]^{2}\). Therefore, the condition (iii) is fulfilled with some  \(\delta\).

%and note that this function is actually a  from the properties of Legendre polynomials \eqref{pn5} and \eqref{scor}, we conclude
%\[ \sum_{j=0}^{J}\left( 2j+1\right) P_{j}(x)P_{j}(y) \geq 
%- (J+1)^{2},
%	\quad
% 	 \left(\sum_{j=0}^{J}\left(
%2j+1\right) P_{j}^{2}(y)\right) ^{1/2} \leq J+1.
%\]

\textbf{4.} In the case of wavelets, we are not able to use Theorem~\ref{propcorcor}, because the trajectories of the process are discontinuous, but can arrive at desired result using some straightforward calculations.  In fact, similar to 
Section~\ref{wav}, we get that 
\begin{eqnarray*}
	\zeta = 2^{l/2} \cdot \left(
		|Z_{0}| + |Z_{1}|
	\right).
\end{eqnarray*}
Hence the density function of \(\zeta/2^{l/2}\) is equal to
\[ 
	p_{|Z_{0}| + | Z_{1} |} (x) = \frac{1}{\pi}
	e^{-x^{2}/4} \int_{-x}^{x} e^{-u^{2}/4} du,
\]
we get that as \(u \to \infty\), 
\begin{eqnarray*}
	\P \left\{
		\zeta>u
	\right\} 
	&\asymp& \frac{2}{\sqrt{\pi}}\int_{u}^{\infty} e^{-u^{2}/2} du =
	4 \left( 1 - \Phi(u/\sqrt{2})\right)\\
	&=& 	\frac{4}{\sqrt{ \pi} x} e^{-x^{2}/2}
	\left(
		1 - \frac{2}{x^2} + 
			o\left( 
				\frac{1}{x^2}
			\right)
		\right),
\end{eqnarray*}
where the last equality follows from  \eqref{aphi}. Comparing the last expression with \eqref{urg}, we arrive at desired result.

\subsection{Proof of Theorem~\ref{main2}}
\label{proofmain2}
From  \eqref{Fst1}  we get that for any \(u\), 
\begin{eqnarray*}
	\P \left\{ 
		\sqrt{\frac{T}{m}} \Z_{n} \leq u
	\right\} 
 	\leq
\left(
			\breve{F}\left(
				\sqrt{m} \; u + c_1 \Lambda_n
			\right)
		\right)^{m}
		+c_2 n^{-\lambda}.
\end{eqnarray*}		
From \eqref{zetaas2}, it follows that for \(u \to \infty\) as \(m \to \infty\), 
\begin{eqnarray*}
1-\breve{F}\left(
	\sqrt{m} u +c_1 \Lambda_n
\right)
\asymp 
2 \frac{h_{1}}{  u^{\k}} 
		\exp\left\{
			-h_{2} \;  u^{2}
		\right\}
\to 0, \qquad u \to \infty.
\end{eqnarray*}	
Therefore, substituting \(u=\tilde{u}_m=u_m - c_1 \Lambda_n m^{-1/2}\),
\begin{eqnarray*}		
	\left(	
			\breve{F}\left(
				\sqrt{m} \; \tilde{u}_m + c_1 \Lambda_n
			\right)
		\right)^{m}
=
		e^{
			m \cdot\ln(1-(1-\breve{F}(\sqrt{m} u_m))
}=
e^{W_{m}}e^{o(W_{m})}, \quad m\to \infty,
\end{eqnarray*}
where
\(
W_{m} :=
-m\cdot\P\left\{
	\zeta^{J,m} 
\geq
	\sqrt{m} u_m 	\right\}\). Applying \eqref{zetaas2} once more, we conclude
\begin{eqnarray*}
	W_{m} =
	- 2 \frac{h_{1} \; m}{  
		u_m^{\k} 
	} 
		\exp\left\{
			- h_2 u_m^2
		\right\} \left( 1 + \btau \left( u_m\right) \right).
\end{eqnarray*}
Since \(a_{m}  \to \infty, b_{m}/c_{m} \to \infty\) as \(m \to \infty\), we get 
\begin{eqnarray*}
W_{m} &=& 
		- 2 \frac{h_{1} \; m}{  b_{m}^{\k} } 
		\exp\left\{
			-h_{2} \cdot \left( 
				2 y \frac{b_{m}}{a_{m} } 
			+
			b_{m}^{2} - 2 c_{m}
						\right)
		\right\} 
		\left( 
			1 + R(m)
		\right),
	\end{eqnarray*}
	because 
	\begin{multline*}
		\left( 
			1 +
			 \frac{y}{a_m b_m}  
			 -
			 \frac{c_m}{b_m^2}
		\right)^{-k}
		\exp \left\{
			-h_2 \left(
				\frac{y^2}{a_m^2} - 2 y \frac{c_m}{a_m b_m} + \frac{c_m^2}{b_m^2}
			\right)
		\right\} \\ 
		=
		\Bigl( 
			1 - \frac{k}{2 h_2} 
			 \frac{\ln b_m}{b_m} 
		\Bigr)
		\Bigl( 
			1 - \frac{k^2}{4 h_2} \left(
			 	\frac{\ln b_m}{b_m}
			\right)^2
		\Bigr)
		\Bigl( 
			1+ o(1)
		\Bigr)\\
		=
		1 - 
		\frac{k}{4 \sqrt{h_2} } 
		\frac{\ln \ln m}{\sqrt{\ln m}} 
		\left( 
			1 + o(1)
		\right) 
		\end{multline*}
Finally, we get the following asympotics for \(W_m\):
	\begin{eqnarray*}
	W_{m}
		&=&
		 - 2 \exp\left\{
			- 2 h_{2} y 
			\frac{b_{m}}{a_{m}} 
		\right\}
		\cdot
		\frac{
			h_{1} \; m
		}{ 
			\exp\left\{
				h_{2}b_{m}^{2}
			\right\}
		}
		\cdot
		\frac{
		\exp\left\{
				2 h_{2}c_{m}
			\right\}
		}{ b_{m}^{k}
		} (1 +R(m)) \\
		&=& - 2 e^{-y} (1 + R(m)),
\end{eqnarray*}
which leads to the conclusion that for any \(y \in \R\),
\begin{multline}
\label{pp}
 	\P \left\{
		\sqrt{\frac{T}{m}}
		\sup_{x \in D} \left(	
			\frac{
				\left|
					\hs_n (x)  - \E\hs_n (x)  
				\right|
			}
			{	
				\sqrt{s(x)}
			}
		\right)
		\leq u_{m} - c_1 \frac{\Lambda_n }{\sqrt{m}}
	\right\} 
	\\ \leq
	 e^{-2 e^{-y}} \left( 
			1 + R(m)
		\right)+c_2 n^{-\lambda}.
\end{multline}
Note that we use the notation \(R(m)\) for all functions which can be represented in the form \eqref{s} with possibly different functions \(o(1)\).
Hence \(u_{m}=u_m(y)\) can be represented as 
\(
u_{m}(y)=y (2h_{2} b_{m})^{-1}-c_{m} b_{m}^{-1}+b_{m}, 
\)%
we get that 
\begin{equation*}
u_{m}(y) +  c_{1}\frac{\Lambda _{n}}{n}=
u_{m}\left(y + 2c_{1}h_{2}b_{m}\frac{%
\Lambda _{n}}{n}\right). 
\end{equation*}%
Moreover, 
\begin{eqnarray*}
2c_{1}h_{2}b_{m}\frac{\Lambda_{n}}{n}&=&
2c_{1}h_{2}  \sqrt{\frac{\Lambda_{n}^{2}}{n^{2}}\ln \left( \frac{\Lambda
_{n}n^{\varkappa /2}}{\ln n}\right) }\\
&=&2c_{1}h_{2} \sqrt{\frac{\Lambda_{n}^{2}\ln
\Lambda _{n}}{n^{2}}+\frac{\varkappa \Lambda _{n}^{2}\ln n}{2n^{2}}-\frac{%
\Lambda _{n}^{2}\ln \ln n}{n^{2}}}=o\left( \frac{\sqrt{\ln n}}{n}\right).\end{eqnarray*}
Since \eqref{pp} is fulfilled uniformly over the compact sets, we are able to apply this inequality with \(y - 2c_{1}h_{2}b_{m}\Lambda _{n}/ n\) instead of \(y\). Finally, we arrive at 
\begin{eqnarray}
\label{pp2}
 	\P \left\{
		\sqrt{\frac{T}{m}}
		\sup_{x \in D} \left(	
			\frac{
								\left|
					\hs_n (x)  - \E\hs_n (x)  
				\right|
			}
			{	
				\sqrt{s(x)}
			}
		\right)
		\leq u_{m}
	\right\} 
	\leq
	 e^{-2 e^{-y}} \left( 
			1 + R(m)
		\right).
\end{eqnarray}
Analogously to \eqref{pp} and \eqref{pp2}, we derive from \eqref{Fst2} that 
\begin{multline}
\label{pp3}
\P \left\{
		\sqrt{\frac{T}{m}}
		\sup_{x \in D} \left(	
			\frac{
								\left|
					\hs_n (x)  - \E\hs_n (x)  
				\right|
			}
			{	
				\sqrt{s(x)}
			}
		\right)
		\leq u_{m} + c_1 \frac{\Lambda_n }{\sqrt{m}}
	\right\} \\
	\geq e^{-2 e^{-y}} \left( 
			1 + R(m)
		\right)-c_2 n^{-\lambda},
\end{multline}
and
\begin{eqnarray}
		\label{pp4}
		\P \left\{
		\sqrt{\frac{T}{m}}
		\sup_{x \in D} \left(	
			\frac{
								\left|
					\hs_n (x)  - \E\hs_n (x)  
				\right|
			}
			{	
				\sqrt{s(x)}
			}
		\right)
		\leq u_{m} 
	\right\} 
	&\geq& e^{-2 e^{-y}} \left( 
			1 + R(m)
		\right).
\end{eqnarray}
Joint consideration of \eqref{pp2} and \eqref{pp4}  completes the proof.
\subsection{Proof of Theorem~\ref{main3}}
\label{A3}
Consider the difference 
\begin{eqnarray*}
\E \hs_n (x)  - s(x)  &= &\sum_{r=1}^d \left[
	\frac{1}{\Delta} 
	\E \left( 
		\varphi_r (X_\Delta)
	\right)
	-
	\int_a^b \varphi_r (u) s(u) du 
	\right]
	\varphi_r (x) \\
	&= &\sum_{j=0}^J \sum_{p=1}^m \left[
	\frac{1}{\Delta} 
	\E \left( 
		\varphi_{j,p} (X_\Delta)
	\right)
	-
	\int_{I_p} \varphi_{j,p} (u) s(u) du 
	\right]
	\varphi_{j,p} (x),
\end{eqnarray*}
where \(\varphi_{j,p}(x) =  \psi_{j}^m  \left(  x -  \delta (p-1)  \right) I\left\{ x \in I_{p} \right\}, 
  j=0..J,\; p=1..m\).
Due to Lemma~B.1 from \cite{Fig3},
\begin{multline*}
\left|
	\frac{1}{\Delta} 
	\E \left( 
		\varphi_{j,p} (X_\Delta)
	\right)
	-
	\int_{I_p} \varphi_{j,p} (u) s(u) du 
\right| \\ \leq 
\left( 
	\left|
		\psi_j^m (a)
	\right|
	+
	\int_a^{a+\delta}
		\left| 
			(\psi_j^m)' (u) 
		\right|
	du
\right) M_\Delta (I_p),
\end{multline*}
where 
\begin{eqnarray*}
	M_\Delta (A) = \sup_{y \in A} \left| 
		\frac{1}{\Delta} \P \left\{ 
			X_{\Delta} > y 
		\right\} 
		-
		\nu\left(
			[y, +\infty)
		\right)
	\right|, \qquad A \subset \R.
\end{eqnarray*}
Applying the small-time asymptotic result (2.2) from \cite{Fig3}, which we also use in this paper (see \eqref{supp}), we conclude that \(M_\Delta (I_p) \leq M_\Delta ([a,b]) \leq q \Delta\), and therefore 
\begin{multline*}
	\sup_{x \in [a,b]}
	\left|
		\E \hs_n (x)  - s(x) 
	\right| \\
	\leq
	q n^{\kappa-1} \cdot
	\sum_{j=0}^{J} 
	\left( 
		| \psi_{j}^m (a) |
		+
		\int_{a}^{a+\delta} \left| 
			(\psi_{j}^m)' (u)
		\right|
		du 
	\right)
	\sup_{x \in [a,a+\delta]}
	\left|
		\psi_{j}^m (x) 
	\right|.
\end{multline*}
Next, using the conditions \eqref{cond1}, we arrive at 
\begin{eqnarray}
\label{bias}
\sup_{x \in [a,b]}
	\left|
		\E \hs_n (x)  - s(x) 
	\right|
	\leq \cc n^{\kappa-1} m.
\end{eqnarray}
Since 
\begin{multline}
\label{f64}
		\sqrt{\frac{T}{m}}
		\sup_{x \in D} \left(	
			\frac{
				\left|
					\hs_n (x)  - s (x)  
				\right|
			}
			{	
				\sqrt{s(x)}
			}
		\right)
		 \\ 
		 \leq 
			\sqrt{\frac{T}{m}}
		\sup_{x \in D} \left(	
			\frac{
				\left|
					 \hs_n (x)  - \E \hs_n (x)  
				\right|
			}
			{	
				\sqrt{s(x)}
			}
		\right)
		+	
					\sqrt{\frac{T}{m}}
		\sup_{x \in D} \left(	
			\frac{
				\left|
					\E \hs_n (x)  - s (x)  
				\right|
			}
			{	
				\sqrt{s(x)}
			}
		\right)\\
		\leq u_m+  \cc \cdot  n^{(3\kappa/2)-1} m^{1/2},
\end{multline}
we arrive at the desired result.

\bibliographystyle{imsart-nameyear}
\bibliography{Panov_bibliography}

\end{document}